\documentclass{amsart}
\usepackage{amsmath, amssymb, amsthm, mathrsfs, amscd}
\usepackage{graphicx}

\theoremstyle{plain} 
 \newtheorem{thm}{Theorem}[section]
 \newtheorem{lem}[thm]{Lemma}
 
 \newtheorem{prop}[thm]{Proposition}
 \newtheorem{claim}[thm]{Claim}

\theoremstyle{definition}
  \newtheorem{defn}[thm]{Definition}

\theoremstyle{remark}
  \newtheorem{rem}[thm]{Remark}

\newcommand{\comm}{{\rm Comm}}
\renewcommand{\mod}{{\rm Mod}}
\renewcommand{\pmod}{{\rm PMod}}
\newcommand{\aut}{{\rm Aut}}
\newcommand{\cal}{\mathcal}
\renewcommand{\frak}{\mathfrak}

\newcommand{\cali}{\mathcal{I}}
\newcommand{\calk}{\mathcal{K}}
\newcommand{\calc}{\mathcal{C}}
\newcommand{\calt}{\mathcal{T}}
\newcommand{\lk}{{\rm Lk}}

\begin{document}

\title[The Torelli complex and the complex of separating curves]{Automorphisms of the Torelli complex and \\the complex of separating curves}
\author{Yoshikata Kida}
\address{Department of Mathematics, Kyoto University, 606-8502 Kyoto, Japan}
\email{kida@math.kyoto-u.ac.jp}
\date{September 25, 2009, revised on December 16, 2010}
%start: July 21, 2009

\begin{abstract}
We compute the automorphism groups of the Torelli complex and the complex of separating curves for all but finitely many compact orientable surfaces.
As an application, we show that the abstract commensurators of the Torelli group and the Johnson kernel for such surfaces are naturally isomorphic to the extended mapping class group.
\end{abstract}

\maketitle

%%%%%%%%%%%%%%%%%%%%%%%%%%%%%%%%%%%%%%%%%%%%% 

\section{Introduction}

Let $S=S_{g, p}$ denote a connected, compact and orientable surface of genus $g$ with $p$ boundary components. 
Unless otherwise stated, we assume a surface to satisfy these conditions.
The complex of curves for $S$, denoted by $\calc(S)$, plays an important role in the study of the mapping class group $\mod(S)$ for $S$.
In fact, understanding of automorphisms of $\calc(S)$ leads to the computation of the commensurator of $\mod(S)$ as discussed in \cite{iva-aut} and \cite{kork-aut} (see also \cite{luo} for automorphisms of $\calc(S)$).
The aim of this paper is to compute automorphisms of the Torelli complex for $S$ and the complex of separating curves for $S$, denoted by $\calt(S)$ and $\calc_s(S)$, respectively.
These simplicial complexes are variants of the complex of curves.
When $S$ is closed, they are used in \cite{farb-ivanov} and \cite{bm} to compute the commensurators of the Torelli group and the Johnson kernel for $S$ (see also \cite{mv} for a related work).
We have natural simplicial actions of the extended mapping class group $\mod^*(S)$ for $S$ on $\calt(S)$ and on $\calc_s(S)$.
The following theorems show that the induced homomorphisms from $\mod^*(S)$ into the automorphism groups $\aut(\calt(S))$ and $\aut(\calc_s(S))$ are isomorphisms for all but finitely many surfaces $S$.
As an application, we prove that the commensurators of the Torelli group $\cali(S)$ and the Johnson kernel $\calk(S)$ for $S$ are naturally isomorphic to $\mod^*(S)$.
We refer to Section \ref{sec-comp} for a precise definition of simplicial complexes and groups mentioned above.

We recall a definition of the commensurator of a group $\Gamma$. 
Let $F(\Gamma)$ be the set of all isomorphisms between finite index subgroups of $\Gamma$.
We say that two elements $f$, $h$ of $F(\Gamma)$ are equivalent if there exists a finite index subgroup of $\Gamma$ on which $f$ and $h$ are equal. 
The composition of two elements $f\colon \Gamma_1\rightarrow \Gamma_2$, $h\colon \Lambda_1\rightarrow \Lambda_2$ of $F(\Gamma)$ given by $f\circ h\colon h^{-1}(\Gamma_1\cap \Lambda_2)\rightarrow f(\Lambda_2\cap \Gamma_1)$ induces the product operation on the quotient set of $F(\Gamma)$ by this equivalence relation. 
This makes it into a group, called the {\it (abstract) commensurator} of $\Gamma$ and denoted by $\comm(\Gamma)$. 
Let ${\bf i}\colon \Gamma \rightarrow \comm(\Gamma)$ denote the homomorphism defined by inner conjugation. 
We note that ${\bf i}$ is injective if and only if the center of any finite index subgroup of $\Gamma$ is trivial.

The following theorem computes automorphisms of the Torelli complex and the commensurator of the Torelli group.

\begin{thm}\label{thm-tor-comm}
Let $S=S_{g, p}$ be a surface and assume one of the following three conditions: $g=1$ and $p\geq 3$; $g=2$ and $p\geq 2$; or $g\geq 3$ and $p\geq 0$. 
Then
\begin{enumerate}
\item the homomorphism from $\mod^*(S)$ into $\aut(\calt(S))$ is an isomorphism.
\item the homomorphism ${\bf i}\colon \mod^*(S)\rightarrow \comm(\cali(S))$ defined by conjugation is an isomorphism.
\end{enumerate}
\end{thm}

We note that Farb-Ivanov \cite{farb-ivanov} announce the computation of automorphisms of the Torelli geometry for a closed surface, which is the Torelli complex with a certain marking.
As a consequence of it, they also announce Theorem \ref{thm-tor-comm} (ii) for $S=S_{g, 0}$ with $g\geq 5$.
McCarthy-Vautaw \cite{mv} compute automorphisms of $\cali(S)$ for $S=S_{g, 0}$ with $g\geq 3$.
Brendle-Margalit \cite{bm}, \cite{bm-add} obtain Theorem \ref{thm-tor-comm} for $S=S_{g, 0}$ with $g\geq 3$.
An analogous result on the complex of separating curves and the Johnson kernel is the following:

\begin{thm}\label{thm-jo-comm}
Let $S$ be the surface in Theorem \ref{thm-tor-comm}. 
Then
\begin{enumerate}
\item the homomorphism from $\mod^*(S)$ into $\aut(\calc_s(S))$ is an isomorphism.
\item the homomorphism ${\bf i}\colon \mod^*(S)\rightarrow \comm(\calk(S))$ defined by conjugation is an isomorphism.
\end{enumerate}
\end{thm}

Theorem \ref{thm-jo-comm} for closed surfaces is due to Brendle-Margalit \cite{bm}, \cite{bm-add}. 
Our proof of Theorems \ref{thm-tor-comm} and \ref{thm-jo-comm} for surfaces of genus at least two partly follows their argument using sharing pairs and spines. 
On the other hand, if $S$ is a surface of genus one, then there exists no sharing pair in $S$. 
We examine the case $S=S_{1, 3}$ through pentagons in $\calt(S)$ and hexagons in $\calc_s(S)$, which are cycles of length five and six, respectively.
Afterward, we prove Theorems \ref{thm-tor-comm} and \ref{thm-jo-comm} for $S=S_{1, p}$ with $p\geq 3$ by induction on $p$.

\begin{rem}
Let us describe several facts on surfaces which are not dealt with in Theorems \ref{thm-tor-comm} and \ref{thm-jo-comm}. 

If $S=S_{0, p}$ is a surface of genus zero with $p\geq 5$, then both $\calt(S)$ and $\calc_s(S)$ are equal to $\calc(S)$, and both $\cali(S)$ and $\calk(S)$ are equal to the pure mapping class group $\pmod(S)$ of $S$. 
The same conclusions as Theorems \ref{thm-tor-comm} and \ref{thm-jo-comm} therefore hold for such an $S$ due to \cite{kork-aut}.

The Birman exact sequence shows that $\cali(S_{1, 2})$ is isomorphic to $\pi_1(S_{1, 1})$ and $\calk(S_{1, 2})$ is isomorphic to the commutator subgroup $[\pi_1(S_{1, 1}), \pi_1(S_{1, 1})]$. 
Mess \cite{mess} proved that $\cali(S_{2, 0})=\calk(S_{2, 0})$ is isomorphic to the free group of infinite rank. 
Both $\calt(S)$ and $\calc_s(S)$ are zero-dimensional if $S=S_{1, 2}$ or $S_{2, 0}$. 
It follows that the same conclusions as Theorems \ref{thm-tor-comm} and \ref{thm-jo-comm} are not true for $S_{1, 2}$ and $S_{2, 0}$.

As for $S_{2, 1}$, we know that the homomorphism from $\mod^*(S_{2, 1})$ into $\aut(\calc_s(S_{2, 1}))$ is not surjective.
Indeed, $\calc_s(S_{2, 1})$ consists of countably infinitely many $\aleph_0$-regular trees. 
This is a consequence of the following facts:
\begin{itemize}
\item Let $S=S_{2, 1}$ be a surface, and let $\bar{S}$ be the surface obtained by attaching a disk to the boundary of $S$.
Let $\pi \colon \calc_s(S)\rightarrow \calc_s(\bar{S})$ be the simplicial map associated with the inclusion of $S$ into $\bar{S}$. 
The fiber of $\pi$ over each vertex of $\calc_s(\bar{S})$ is then a tree (see Theorem 7.1 of \cite{kls}).
\item $\calc_s(S_{2, 0})$ is a zero-dimensional simplicial complex consisting of countably infinitely many vertices.
\end{itemize}
On the other hand, $\calt(S_{2, 1})$ is a connected graph, which contains a hexagon because one can embed $\calc_s(S_{1, 3})$ into $\calt(S_{2, 1})$ by gluing any two boundary components of $S_{1, 3}$ (see Figure \ref{fig-hex} for a hexagon in $\calc_s(S_{1, 3})$).
\end{rem}

More generally, we study superinjective maps from the Torelli complex $\calt(S)$ into itself when $S$ is a surface of genus one. 
Superinjectivity of simplicial maps from $\calc(S)$ into itself was introduced by Irmak \cite{irmak1}, \cite{irmak2} to study injective homomorphisms from finite index subgroups of $\mod^*(S)$ into $\mod^*(S)$ (see \cite{be-m}, \cite{bm-ar}, \cite{irmak-ns} and \cite{sha} for related works). 
Superinjectivity of simplicial maps from $\calt(S)$ into itself is also defined similarly (see Section \ref{subsec-comp}). 
As a result, we obtain the following:

\begin{thm}\label{thm-g-1}
Let $S=S_{1, p}$ be a surface with $p\geq 3$. Then
\begin{enumerate}
\item any superinjective map from $\calt(S)$ into itself is induced by an element of $\mod^*(S)$.
\item if $\Gamma$ is a finite index subgroup of $\cali(S)$ and if $f\colon \Gamma \rightarrow \cali(S)$ is an injective homomorphism, then there exists an element $\gamma_0$ of $\mod^*(S)$ with the equality $f(\gamma)= \gamma_0\gamma \gamma_0^{-1}$ for any $\gamma \in \Gamma$. 
In particular, $\Gamma$ is co-Hopfian.
\end{enumerate}
\end{thm}

Recall that a group $\Gamma$ is said to be {\it co-Hopfian} if any injective homomorphism from $\Gamma$ into itself is surjective.

The paper is organized as follows. 
In Section \ref{sec-comp}, we collect definitions of simplicial complexes and groups mentioned above. 
Fundamental properties of them are also reviewed. 
In Section \ref{sec-basic-torelli}, we describe simplices of $\calt(S)$ of maximal dimension and observe topological information on vertices of $\calt(S)$ preserved by a superinjective map $\phi$ from $\calt(S)$ into itself. 
In Section \ref{sec-super-torelli}, when $S$ is a surface of genus one, we construct a simplicial map $\Phi$ from $\calc(S)$ into itself inducing $\phi$.
Applying the result due to \cite{sha} on injective simplicial maps from $\calc(S)$ into itself, we show that $\Phi$ is an automorphism of $\calc(S)$ and obtain Theorem \ref{thm-g-1} (i). 
In Section \ref{sec-aut-sep}, when $S$ is a surface in Theorem \ref{thm-tor-comm}, for any automorphism of $\calc_s(S)$, we construct an automorphism of $\calc(S)$ extending it.
As a consequence of it, for any automorphism of $\calt(S)$, we construct an automorphism of $\calc(S)$ inducing it. 
The argument for surfaces of genus at least two depends on \cite{bm}. 
In Section \ref{sec-twist}, we present an algebraic characterization of twisting elements of $\cali(S)$, following Vautaw's argument for closed surfaces in \cite{v} and \cite{v-t}.
We then associate an automorphism of $\calt(S)$ (resp.\ $\calc_s(S)$) to each isomorphism between finite index subgroups of $\cali(S)$ (resp.\ $\calk(S)$).
As a result, we compute the commensurators of $\cali(S)$ and $\calk(S)$. 
In Section \ref{sec-torus}, the commensurators of the braid groups on the torus are described by using Theorem \ref{thm-tor-comm} for surfaces of genus one and following argument in \cite{lm}. 
In Appendix, we prove that each element of $\cali(S)$ is pure in the sense of Ivanov \cite{iva-subgr}. 
This fact is used in Section \ref{sec-twist}.

\medskip

\noindent {\it Acknowledgements.} The manuscript of this paper was written during the stay at Institut des Hautes \'Etudes Scientifiques. 
The author thanks the institute for giving nice environment and for warm hospitality.

%%%%%%%%%%%%%%%%%%%%%%%%%%%%%%%%%%%%%%%%%%%%%

\section{Complexes and groups associated with a surface}\label{sec-comp}

\subsection{Terminology}

Unless otherwise stated, we assume that a surface is connected, compact and orientable, and it may have non-empty boundary. 
Let $S=S_{g, p}$ be a surface of genus $g$ with $p$ boundary components. 
The Euler characteristic of $S$ is denoted by $\chi(S)$ and is equal to $-2g-p+2$. 
A simple closed curve in $S$ is said to be {\it essential} in $S$ if it is neither homotopic to a point of $S$ nor isotopic to a boundary component of $S$.
We denote by $V(S)$ the set of isotopy classes of essential simple closed curves in $S$.
Let $i\colon V(S)\times V(S)\rightarrow \mathbb{Z}_{\geq 0}$ denote the {\it geometric intersection number}, i.e., the minimal cardinality of the intersection of representatives for two elements of $V(S)$.
Given $\alpha, \beta \in V(S)$ and their representatives $A$, $B$, respectively, we say that $A$ and $B$ {\it intersect minimally} if we have $|A\cap B|=i(\alpha, \beta)$.

When there is no confusion, we mean by a curve in $S$ either an essential simple closed curve in $S$ or its isotopy class.
A curve $a$ in $S$ is said to be {\it separating} in $S$ if $S\setminus a$ is not connected, and otherwise $a$ is said to be {\it non-separating} in $S$.
Whether an essential simple closed curve in $S$ is separating in $S$ or not depends only on its isotopy class.
A pair of non-separating curves in $S$, $\{ a, b \}$, is called a {\it bounding pair (BP)} in $S$ if $a$ and $b$ are disjoint and non-isotopic and if $S\setminus (a \cup b)$ is not connected.
These conditions depend only on the isotopy classes of $a$ and $b$.
When we take into account an order of the two curves of a BP $\{ a, b \}$, it is denoted by $(a, b)$ and is called an {\it ordered bounding pair}. 
We often confuse a BP with and without an order if they can be distinguished in the context. 
We say that two non-separating curves in $S$ are {\it BP-equivalent} in $S$ if they either are isotopic or are disjoint and form a BP in $S$.

We mean by a {\it handle} a surface homeomorphic to $S_{1, 1}$ and mean by a {\it pair of pants} a surface homeomorphic to $S_{0, 3}$.
Let $a$ be a separating curve in $S$. 
If $a$ cuts off a handle from $S$, then $a$ is called an {\it h-curve} in $S$.
If $a$ cuts off a pair of pants from $S$, then $a$ is called a {\it p-curve} in $S$.
A curve which is either an h-curve or a p-curve in $S$ is called an {\it hp-curve} in $S$.

%%%%%%%%%%%%%%%%%%%%%%%%%%%%%%%%%%%%%%%%%%%%%

\subsection{The complex of curves and its variants}\label{subsec-comp}

We collect definitions of three abstract simplicial complexes associated with simple closed curves in surfaces.
The complex of curves was introduced by Harvey \cite{harvey}.
The complex of separating curves appears in \cite{farb-ivanov}, \cite{mv}, \cite{bm} and \cite{bm-add}.
The Torelli complex (with a certain marking and for a closed surface) was introduced by Farb-Ivanov \cite{farb-ivanov}.
We fix a surface $S$.

\medskip

\noindent {\bf The complex of curves.} Let $\Sigma(S)$ denote the set of non-empty finite subsets $\sigma$ of $V(S)$ with $i(\alpha, \beta)=0$ for any $\alpha, \beta \in \sigma$.
The {\it complex of curves} for $S$, denoted by $\calc(S)$, is defined as the abstract simplicial complex such that the sets of vertices and simplices of it are $V(S)$ and $\Sigma(S)$, respectively.

\medskip

\noindent {\bf The complex of separating curves.} Let $V_s(S)$ denote the subset of $V(S)$ consisting of isotopy classes of separating curves in $S$.
The {\it complex of separating curves} for $S$, denoted by $\calc_s(S)$, is defined as the full subcomplex of $\calc(S)$ spanned by $V_s(S)$.

\medskip

We extend the geometric intersection number $i$ to the symmetric function on $(V(S)\sqcup \Sigma(S))^2$ so that $i(\alpha, \sigma)=\sum_{\beta \in \sigma}i(\alpha, \beta)$ and $i(\sigma, \tau)=\sum_{\beta \in \sigma, \gamma \in \tau}i(\beta, \gamma)$ for any $\alpha \in V(S)$ and $\sigma, \tau \in \Sigma(S)$.
We say that two elements $\sigma$, $\tau$ of $V(S)\sqcup \Sigma(S)$ are {\it disjoint} if $i(\sigma, \tau)=0$, and otherwise we say that they {\it intersect}.

\medskip

\noindent {\bf The Torelli complex.} Let $V_{bp}(S)$ denote the set of isotopy classes of BPs in $S$.
Each element of $V_{bp}(S)$ is often regarded as an edge of $\calc(S)$.
We define $V_t(S)$ as the disjoint union $V_s(S)\sqcup V_{bp}(S)$.
The {\it Torelli complex} for $S$, denoted by $\calt(S)$, is the abstract simplicial complex such that the set of vertices is $V_t(S)$ and a non-empty finite subset $\sigma$ of $V_t(S)$ is a simplex of $\calt(S)$ if and only if any two elements of $\sigma$ are disjoint.
Let $\Sigma_t(S)$ denote the set of simplices of $\calt(S)$.

\medskip

We note that if $S$ is a surface of genus zero, then both $\calc_s(S)$ and $\calt(S)$ are equal to $\calc(S)$ since any essential simple closed curve in $S$ is separating in $S$.

Let us collect here terminology and symbols used throughout this paper. 
Pick $\sigma \in \Sigma(S)$. 
A {\it BP-equivalence class} in $\sigma$ is an equivalence class in the set of all non-separating curves in $\sigma$ with respect to the BP-equivalence relation. 
When all curves of $\sigma$ are non-separating in $S$ and BP-equivalent to each other, we say that $\sigma$ {\it forms a BP-equivalence class}. 
Two elements $b_1$, $b_2$ of $V_{bp}(S)$ are said to be {\it BP-equivalent} if $b_1$ and $b_2$ are disjoint and the set of all curves in $b_1$ and $b_2$ forms a BP-equivalence class. 
An element of $V_{bp}(S)$ is called a {\it BP-vertex}. 
Similarly, an element of $V(S)$ corresponding to an h-curve and a p-curve in $S$ is called an {\it h-vertex} and a {\it p-vertex}, respectively.

Let $X$ be one of the simplicial complexes $\calc(S)$, $\calc_s(S)$ and $\calt(S)$.
We denote by $V(X)$ the set of vertices of $X$. 
Note that a map $\phi \colon V(X)\rightarrow V(X)$ defines a simplicial map from $X$ into itself if and only if we have $i(\phi(a), \phi(b))=0$ for any two vertices $a, b\in V(X)$ with $i(a, b)=0$. 
We mean by a {\it superinjective map} $\phi \colon X\rightarrow X$ a simplicial map $\phi \colon X\rightarrow X$ satisfying $i(\phi(a), \phi(b))\neq 0$ for any two vertices $a, b\in V(X)$ with $i(a, b)\neq 0$. 

Any superinjective map $\phi \colon X\rightarrow X$ is in fact injective.
For if there were two distinct vertices $a, b\in V(X)$ with $\phi(a)=\phi(b)$, then superinjectivity of $\phi$ would imply $i(a, b)=0$.
Since we have $a\neq b$, there exists a vertex $c\in V(X)$ with $i(a, c)=0$ and $i(b, c)\neq 0$.
This is a contradiction because we have $i(\phi(a), \phi(c))=0$ and $i(\phi(b), \phi(c))\neq 0$ by superinjectivity of $\phi$.

For each $\sigma \in \Sigma(S)$, we denote by $S_{\sigma}$ the surface obtained by cutting $S$ along all curves in $\sigma$. 
When $\sigma$ consists of a single curve $a$, we denote it by $S_a$ for simplicity.
We often identify a component of $S_{\sigma}$ with a complementary component in $S$ of a tubular neighborhood of a one-dimensional submanifold representing $\sigma$ if there is no confusion.
If $Q$ is a component of $S_{\sigma}$, then $V(Q)$ is naturally identified with a subset of $V(S)$.

%%%%%%%%%%%%%%%%%%%%%%%%%%%%%%%%%%%%%%%%%%%%%

\subsection{The mapping class group and its subgroups}\label{subsec-mcg}

Let $S$ be a surface. 
The {\it extended mapping class group} $\mod^*(S)$ for $S$ is the group consisting of all isotopy classes of homeomorphisms from $S$ onto itself, where isotopy may move points in the boundary of $S$. 
The {\it mapping class group} $\mod(S)$ for $S$ is the subgroup of $\mod^*(S)$ consisting of all isotopy classes of orientation-preserving homeomorphisms from $S$ onto itself. 
The {\it pure mapping class group} $\pmod(S)$ for $S$ is the subgroup of $\mod(S)$ consisting of all isotopy classes of orientation-preserving homeomorphisms from $S$ onto itself that fix each boundary component of $S$ as a set. 
The reader should consult \cite{fm}, \cite{flp} and \cite{iva-mcg} for fundamentals of these groups.

Given the isotopy class $a$ of an essential simple closed curve in $S$, we denote by $t_a\in \pmod(S)$ the {\it (left) Dehn twist} about $a$. 
For an ordered BP $x=(a, b)$, we write $t_x=t_a t_b^{-1}$ and call it the {\it BP twist} about $x$. 
The {\it Torelli group} $\cali(S)$ for $S$ is the subgroup of $\pmod(S)$ generated by Dehn twists about all separating curves in $S$ and BP twists about all BPs in $S$. 
The {\it Johnson kernel} $\calk(S)$ for $S$ is the subgroup of $\pmod(S)$ generated by Dehn twists about all separating curves in $S$.
Both $\cali(S)$ and $\calk(S)$ are normal subgroups of $\mod^*(S)$. 
We refer to \cite{putman-tor} for variants of the definition of the Torelli group. 
Note that if $S$ is a surface of genus zero, then any curve in $S$ is separating in $S$, and thus both $\cali(S)$ and $\calk(S)$ are equal to $\pmod(S)$.

Let $S=S_{g, p}$ be a surface. 
Due to Powell \cite{powell} (for closed surfaces and based on Birman's work \cite {birman-siegel} on $Sp(2g, \mathbb{Z})$) and Johnson \cite{johnson}, if $g\geq 2$ and $p=0, 1$, then $\cali(S)$ is equal to the subgroup of $\mod(S)$ consisting of all elements that trivially act on the homology group $H_1(S, \mathbb{Z})$.  
This description of $\cali(S)$ is the original definition of the Torelli group for $S$ with $p=0, 1$. 
Afterward, Johnson \cite {johnson1} produced a finite generating set for $\cali(S)$ consisting of BP twists when $g\geq 3$ and $p=0, 1$. 
In contrast, if $g=2$ and $p=0$, then $\cali(S)$ is not finitely generated. 
Indeed, $\cali(S)$ is isomorphic to the free group of infinite rank (see \cite{mm}, \cite{mess} and \cite{bbm}). 
The following fact on $\calk(S)$ is fundamental and will be used in Section \ref{sec-twist}.

\begin{thm}[\cite{johnson-abe}]\label{thm-johnson}
Let $S=S_{g, p}$ be a surface with $g\geq 2$ and $p\leq 1$. 
Then $\calk(S)$ contains no non-zero power of a BP twist.
\end{thm}

\begin{prop}\label{prop-jo-bp}
If $S=S_{g, p}$ is a surface with either {\rm (a)} $g\geq 2$ and $p\geq 0$; or {\rm (b)} $g=1$ and $p\geq 2$, then $\calk(S)$ contains no non-zero power of a BP twist.
\end{prop}

\begin{proof}
We first assume condition (a) and $p\geq 2$. 
Let $x=(a, b)$ be an ordered BP in $S$. 
If $x$ does not cut off a surface of genus zero, then attach disks to any $p-1$ components of $\partial S$. 
Otherwise, choose one component of $\partial S$ contained in the surface of genus zero cut off by $x$, and attach disks to all of the other components of $\partial S$. 
We then obtain the surface $Q$ homeomorphic to $S_{g, 1}$ and the homomorphism $q\colon \pmod(S)\rightarrow \mod(Q)$. 
Note that $x$ can be seen as an ordered BP in $Q$. 
It follows from $q(\calk(S))=\calk(Q)$ and Theorem \ref{thm-johnson} that no non-zero power of $t_{x}$ is contained in $\calk(S)$.

We next assume condition (b). 
Once the conclusion for $S=S_{1, 2}$ is obtained, the conclusion for the other cases is verified along the argument in the previous paragraph. 
Put $S=S_{1, 2}$ and let $R$ be the surface obtained by attaching a disk to one component of $\partial S$, which is homeomorphic to $S_{1, 1}$. 
We then obtain the Birman exact sequence
\[1\rightarrow \pi_1(R)\stackrel{\iota}{\rightarrow} \pmod(S)\rightarrow \mod(R)\rightarrow 1,\]
and $\pi_1(R)$ is generated by two standard non-separating simple loops $a$, $b$ in $R$. 
By the definition of $\iota$, $\iota(a)$ and $\iota(b)$ are BP twists in $\pmod(S)$. 
Thus, $\iota(\pi_1(R))<\cali(S)$. 
Note that $\iota([a, b])$ is the Dehn twist about a separating curve in $S$, where $[a, b]$ denotes the commutator of $a$ and $b$. 
Since the actions of $\pmod(S)$ on the set of all BPs in $S$ and on the set of all separating curves in $S$ are both transitive, the normal closure of $\iota(a)$ in $\pmod(S)$ is equal to $\cali(S)$. 
Thus, $\iota(\pi_1(R))=\cali(S)$. 
The same kind of argument shows that the image of the commutator subgroup $[\pi_1(R), \pi_1(R)]$ via $\iota$ is equal to $\calk(S)$. 
Since no non-zero power of $\iota(a)$ and its conjugate in $\pmod(S)$ lie in $\iota([\pi_1(R), \pi_1(R)])$, the proposition for $S=S_{1, 2}$ follows.
\end{proof}

The following theorem on the automorphism group $\aut(\calc(S))$ of $\calc(S)$, proved in \cite{iva-aut}, \cite{kork-aut} and \cite{luo}, is a fundamental tool to compute commensurators of mapping class groups and their subgroups.

\begin{thm}\label{thm-cc}
Let $S=S_{g, p}$ be a surface with $3g+p-4>0$.
We define
\[\pi \colon \mod^*(S)\rightarrow \aut(\calc(S))\]
as the homomorphism associated with the natural action of $\mod^*(S)$ on $\calc(S)$. 
Then
\begin{enumerate}
\item if $(g, p)\neq (1, 2), (2, 0)$, then $\pi$ is an isomorphism.
\item if $(g, p)=(1, 2)$, then $\ker \pi$ is equal to the center of $\mod^*(S)$.
The image of $\pi$ is equal to the group of automorphisms of $\calc(S)$ preserving vertices corresponding to a separating curve in $S$, which is a finite index subgroup of $\aut(\calc(S))$.
\item if $(g, p)=(2, 0)$, then $\ker \pi$ is equal to the center of $\mod^*(S)$, and $\pi$ is surjective.
\end{enumerate}
\end{thm}

It is known that if $(g, p)=(1, 2), (2, 0)$, then the center of $\mod^*(S)$ is isomorphic to $\mathbb{Z}/2\mathbb{Z}$. 
Any superinjective map from $\calc(S)$ into itself is shown to be surjective in \cite{be-m}, \cite{bm-ar}, \cite{irmak1}, \cite{irmak2} and \cite{irmak-ns}. 
More generally, the following is obtained.

\begin{thm}[\cite{sha}]\label{thm-sha}
Let $S=S_{g, p}$ be a surface with $3g+p-4>0$. 
Then any injective simplicial map from $\calc(S)$ into itself is surjective.
\end{thm}

%%%%%%%%%%%%%%%%%%%%%%%%%%%%%%%%%%%%%%%%%%%%%

\section{Basics of the Torelli complex}\label{sec-basic-torelli}

In this section, we show that any superinjective map from the Torelli complex into itself preserves the topological types of vertices.
The same property for the complex of separating curves is also obtained.

\subsection{Simplices of maximal dimension}

We describe simplices of the Torelli complex of maximal dimension. 
The following observation on BPs will be used many times in the sequel.

\begin{lem}\label{lem-bp}
Let $a$ be a BP in $S$, and let $b$ be either a separating curve in $S$ disjoint from $a$ or a BP in $S$ which is disjoint from $a$ and is not BP-equivalent to $a$. 
Then the two curves in $a$ are contained in a single component of $S_b$.
\end{lem}

\begin{proof}
Let $a_1$ and $a_2$ denote the two non-separating curves in $a$. 
Assume $b$ to be a separating curve in $S$. 
For each $j=1, 2$, $a_j$ is non-separating in the component of $S_b$ containing it. 
If $a_1$ and $a_2$ were contained in distinct components of $S_b$, then $S\setminus (a_1\cup a_2)$ would be connected. 
This is a contradiction. 
We next assume $b$ to be a BP in $S$. 
Each separating curve in a component of $S_b$ is either separating in $S$ or forms a BP in $S$ with any curve of $b$. 
Since $a$ and $b$ are not BP-equivalent, each curve in $a$ is non-separating in the component of $S_b$ containing it. 
The conclusion of the lemma then follows as before.
\end{proof}

\begin{lem}\label{lem-bc}
Let $S$ be a surface and let $b$ and $c$ be simplices of $\calc(S)$ such that
\begin{itemize}
\item $|b|\geq 2$, $|c|\geq 2$, $b\cap c=\emptyset$ and $b\cup c\in \Sigma(S)$; and
\item each curve of $b$ and $c$ is non-separating in $S$, and each of $b$ and $c$ forms a BP-equivalence class in the simplex $b\cup c$.
\end{itemize}
Then the following assertions hold:
\begin{enumerate}
\item For each component $Q$ of $S_b$, exactly two curves of $b$ correspond to boundary components of $Q$. 
\item There exists a unique component $R$ of $S_b$ with $c\in \Sigma(R)$. 
Moreover, each curve of $c$ is non-separating in $R$, and $c$ forms a BP-equivalence class as an element of $\Sigma(R)$.
\end{enumerate}
\end{lem}

\begin{proof}
If $Q$ is a component of $S_b$, then the number of curves of $b$ corresponding to boundary components of $Q$ is at least two because each curve of $b$ is non-separating in $S$. 
Assume that the number is at least three. 
Choose three curves $b_1$, $b_2$ and $b_3$ of $b$ corresponding to boundary components of $Q$, and take a p-curve $a$ in $Q$ cutting off a pair of pants whose boundary contains $b_1$ and $b_2$. 
Since the pair $\{ b_1, b_2\}$ separates $S$ into two components, $a$ is separating in $S$. 
This contradicts Lemma \ref{lem-bp} because both components of $S_a$ contain a curve of $b$. 
Assertion (i) is proved.

Take a curve $c_1$ of $c$ and a component $R$ of $S_b$ with $c_1\in V(R)$. 
We first claim that $c_1$ is non-separating in $R$. 
Assume otherwise, and let $b_1, b_2\in b$ be the curves corresponding to boundary components of $R$. 
If both $b_1$ and $b_2$ were contained in a single component of $R_{c_1}$, then $c_1$ would be separating in $S$. 
This is a contradiction. 
Otherwise, $b_1$ and $c_1$ form a BP in $S$, and this is also a contradiction. 
The curve $c_1$ is therefore non-separating in $R$.

Suppose that a curve $c_2$ of $c$ distinct from $c_1$ is contained in a component $R'$ of $S_b$ distinct from $R$. 
The curve $c_2$ is then non-separating in $R'$, and $S\setminus (c_1\cup c_2)$ is connected because any two curves of $b$ can be connected by an arc in $S$ which does not intersect $c_1$ and $c_2$. 
Since $\{ c_1, c_2\}$ is a BP in $S$, this is a contradiction and proves that each curve of $c$ is contained in $R$. 
Since any two curves in $c$ are BP-equivalent in $S$, $c$ forms a BP-equivalence class as an element of $\Sigma(R)$.
\end{proof}

\begin{lem}\label{lem-sep-number}
Let $S=S_{g, p}$ be a surface and let $\sigma$ be a simplex of $\calc(S)$ consisting of separating curves in $S$. 
Then the following assertions hold:
\begin{enumerate}
\item The inequality $|\sigma |\leq 2g+p-3$ holds and this equality can be attained. 
In particular, $\dim(\calc_s(S))=2g+p-4$. 
\item If $|\sigma|=2g+p-3$, then $S_{\sigma}$ consists of $g$ handles and $g+p-2$ pairs of pants.
\end{enumerate}
\end{lem}

\begin{proof}
The sum of the genera of components of $S_{\sigma}$ is $g$.
If there exists a component $Q$ of $S_{\sigma}$ which is neither a pair of pants nor a handle, then any separating curve in $Q$ is separating in $S$. 
It follows that once assertion (i) is proved, assertion (ii) follows. 
To prove assertion (i), we may assume that each component of $S_{\sigma}$ is either a pair of pants or a handle. 
The number of components of $S_{\sigma}$ is then equal to $|\chi(S)|=2g+p-2$.
The equality $|\sigma|=2g+p-3$ is obtained by the counting argument on the numbers of boundary components of components of $S_{\sigma}$.
\end{proof}

\begin{prop}\label{prop-max}
Let $S=S_{g, p}$ be a surface with $g\geq 1$ and $g+p\geq 3$. Then
\[\dim(\calt(S))=(g-1)+\left(\begin{matrix}
g+p-1 \\
2
\end{matrix}\right)
-1.\]
Moreover, if $g+p\geq 4$, then for any simplex $\sigma$ of $\calt(S)$ of maximal dimension, there exists a unique simplex $\{ \alpha_1,\ldots, \alpha_{g-1}, \beta_1,\ldots, \beta_{g+p-1}\}$ of $\calc(S)$ such that
\begin{enumerate}
\item[(a)] each of $\alpha_1,\ldots, \alpha_{g-1}$ is an h-curve in $S$;
\item[(b)] each of $\beta_1,\ldots, \beta_{g+p-1}$ is a non-separating curve in $S$, and any two of them are BP-equivalent in $S$; and
\item[(c)] $\sigma$ consists of $\alpha_1,\ldots, \alpha_{g-1}$ and all BPs of two of $\beta_1,\ldots, \beta_{g+p-1}$.  
\end{enumerate}
\end{prop}

\begin{proof}
A simplex $\{ \alpha_1,\ldots, \alpha_{g-1}, \beta_1,\ldots, \beta_{g+p-1}\}$ of $\calc(S)$ satisfying conditions (a) and (b) is easily found.
One can verify that the dimension of the simplex of $\calt(S)$ containing all $\alpha_j$ and all BPs of two of $\beta_1,\ldots, \beta_{g+p-1}$ is equal to the right hand side of the equality in the proposition.
It follows that $\dim (\calt(S))$ is not smaller than this number.

We prove the equality in the proposition by induction on $g+p$. 
When $g+p=3$, we have $(g, p)=(1, 2), (2, 1), (3, 0)$, and the cardinality of BP-vertices in a simplex of $\calt(S)$ is at most one.
We then obtain $\dim(\calt(S))=g-1$.

In what follows, we assume $g+p\geq 4$. 
Let $\sigma$ be a simplex of $\calt(S)$ of maximal dimension. 
Let
\[s=\{ a_1,\ldots, a_k, b_{11},\ldots, b_{1m_1}, b_{21},\ldots, b_{lm_l}\}\]
be the collection of all curves in $\sigma \cap V_s(S)$ and all curves in BPs in $\sigma$ so that
\begin{itemize}
\item each of $a_1,\ldots, a_k$ is separating in $S$; and
\item each of $b_{11},\ldots, b_{1m_1}, b_{21},\ldots, b_{lm_l}$ is non-separating in $S$, and the family $b_j=\{ b_{j1},\ldots, b_{jm_j}\}$ forms a BP-equivalence class in $s$ for each $j$.
\end{itemize}
Since $\dim \sigma$ is maximal, $\sigma$ contains the BP of any two curves in $b_j$ for any $j$.
We thus have
\[|\sigma |=k+\sum_{j=1}^l\left(\begin{matrix}
m_j\\
2
\end{matrix}\right)\]
The following two claims show that $s$ contains exactly one BP-equivalence class.

\begin{claim}
We have $l\leq 1$.
\end{claim}

\begin{proof}
We suppose $l>1$ and deduce a contradiction. 
By Lemma \ref{lem-bc}, there exists a unique component $R$ of $S_{b_1}$ containing all curves of $b_2$. 
After exchanging the indices, we may assume that
\begin{itemize}
\item $b_{11}$ and $b_{12}$ are the two curves of $b_1$ corresponding to boundary components of $R$; and
\item we have numbers $k'$, $l'$ with $0\leq k'\leq k$ and $2\leq l'\leq l$ such that $a_1,\ldots, a_{k'}$ and all curves in $b_2,\ldots, b_{l'}$ form the family of all curves in $s$ belonging to $V(R)$. 
Note that $a_1,\ldots, a_{k'}$ and all BPs of two curves in $b_j$ with $2\leq j\leq l'$ form the simplex $\sigma_R$ of $\calt(R)$.
\end{itemize}
Since $b_{11}$ and $b_{12}$ are contained in a single component of the surface obtained by cutting $R$ along each element of $\sigma_R$, there exists a p-curve $c$ in $R$ disjoint from $\sigma_R$ and cutting off a pair of pants whose boundary contains $b_{11}$ and $b_{12}$. 
Note that $c$ is separating in $S$ and has to be contained in $\sigma_R$ since $\dim \sigma$ is maximal. 
It follows that $c$ is equal to one of $a_1,\ldots, a_{k'}$. 
Let $R'$ be the component of $R_c$ which does not contain $b_{11}$ and $b_{12}$ as its boundary components. 
The simplex $\sigma_R\setminus \{ c\}\in \Sigma_t(R)$ can be seen as an element of $\Sigma_t(R')$.

Let $g_1$ denote the genus of $R$ and $p_1$ denote the number of boundary components of $R$. 
Note that $g_1$ is positive because $b_2$ is contained in $R$. 
Since $R'$ is a surface of genus $g_1$ with $p_1-1$ boundary components, the hypothesis of the induction implies the inequality
\[|\sigma_R\setminus \{ c\}|=(k'-1)+\sum_{j=2}^{l'}\left(\begin{matrix}
m_j\\
2
\end{matrix}\right)\leq (g_1-1)+\left(\begin{matrix}
g_1+p_1-2 \\
2
\end{matrix}\right).\]
Choose a simplex $\{ c_1,\ldots, c_{g_1+p_1-3}\} \in \Sigma(R)$ such that for each $j$, $c_j$ is separating in $R$, and $b_{11}$ and $b_{12}$ are not contained in a single component of $R_{c_j}$. 
We can then find a simplex $\{ d_1,\ldots, d_{g_1}\} \in \Sigma(R)$ consisting of h-curves in $R$ disjoint from any $c_j$. 
After deleting $a_1,\ldots, a_{k'}$ and all curves in $b_2,\ldots, b_{l'}$ from $s$, add all $c_j$ and $d_{j'}$ to it. 
This new collection of curves is denoted by $s_1$ and associates the simplex $\sigma_1\in \Sigma_t(S)$ consisting of all separating curves in $s_1$ and all BPs of two curves in $s_1$. 
We then obtain the equality
\[|\sigma_1 |=k-k'+g_1+\left(\begin{matrix}
m_1+g_1+p_1-3 \\
2
\end{matrix}\right)+\sum_{j=l'+1}^l\left(\begin{matrix}
m_j \\
2
\end{matrix}\right)\]
and the inequality
\[|\sigma_1 |-|\sigma| \geq (m_1-1)(g_1+p_1-3)\]
by using the inequality shown above. 
If $g_1+p_1=3$, then we have $g_1=1$ and $p_1=2$, and $b_2$ cannot be in $R$. 
The right hand side in the last inequality is thus positive. 
This contradicts maximality of $\dim \sigma$.
\end{proof}

\begin{claim}
We have $l>0$.
\end{claim}

\begin{proof}
If $l=0$, then $\sigma$ would consist of separating curves in $S$. Lemma \ref{lem-sep-number} implies that $|\sigma|\leq 2g+p-3$. The inequality
\[(g-1)+\left(\begin{matrix}
g+p-1 \\
2
\end{matrix}\right)-(2g+p-3)=\frac{(g+p-2)(g+p-3)}{2}>0\]
then holds when $g+p\geq 4$. This is a contradiction.
\end{proof}

These two claims show that $s$ is of the form
\[s=\{ a_1,\ldots, a_k, b_1,\ldots, b_m\},\]
where each $a_j$ is separating in $S$ and any two curves in the family $b=\{ b_1,\ldots, b_m\}$ are BP-equivalent in $S$. 
We put $T=S_{b_m}$ and $r=s\setminus \{ b_m\}$. 
Any curve in $r$ is then separating in $T$. It follows from maximality of $\dim \sigma$ that the surface $T_r$ consists of $g-1$ handles and $g+p-1$ pairs of pants. 
Applying Lemma \ref{lem-sep-number} to $T$, we obtain the equality $k+(m-1)=2(g-1)+(p+2)-3$. 
Since any $b_j$ cannot be the boundary of a handle in $T_r$, we have the inequality $k\geq g-1$. 
We thus have $m\leq g+p-1$. 
The inequality $g+p-4\geq 0$ then implies the inequality
\begin{align*}
|\sigma|&=k+\left(\begin{matrix}
m \\
2
\end{matrix}\right)=2g+p-2+\frac{m(m-3)}{2}\\
&\leq 2g+p-2+\frac{(g+p-1)(g+p-4)}{2}= g-1+\left(\begin{matrix}
g+p-1 \\
2
\end{matrix}\right).
\end{align*}  
The equality in the proposition is proved. 
The equality in the last inequality holds if and only if $k=g-1$ and $m=g+p-1$. 
The above argument shows that each $a_j$ is an h-curve in $S$. 
We therefore proved that $s$ satisfies conditions (a), (b) and (c). 
Uniqueness of $s$ is obvious. 
\end{proof}

\begin{lem}\label{lem-bp-s}
Let $S=S_{g, p}$ be a surface with $g\geq 1$. 
Suppose either $g+p\geq 4$ or $(g, p)= (3, 0)$. 
Let $\phi \colon \calt(S)\rightarrow \calt(S)$ be a superinjective map. 
Then the following assertions hold:
\begin{enumerate}
\item If $g\geq 2$, then the inclusions $\phi(V_{bp}(S))\subset V_{bp}(S)$ and $\phi(V_s(S))\subset V_s(S)$ hold.
\item If $g=1$, then the inclusion $\phi(V_{bp}(S))\subset V_{bp}(S)$ holds.
\item If $b_1$ and $b_2$ are disjoint BPs in $S$ and are BP-equivalent, then $\phi(b_1)$ and $\phi(b_2)$ are also BP-equivalent. 
\end{enumerate}
\end{lem}

\begin{proof}
Assertion (ii) is a direct consequence of Proposition \ref{prop-max}. 
Assume $g\geq 2$ and $g+p\geq 4$. 
Proposition \ref{prop-max} also implies that if $a$ is an h-curve in $S$, then $\phi(a)$ is either an h-curve in $S$ or a BP in $S$. 
Note that for any simplex $\sigma$ of $\calt(S)$ of maximal dimension, an element $a$ of $\sigma$ is an h-curve in $S$ if and only if there exists a vertex $b$ of $\calt(S)$ with $i(a, b)\neq 0$ and $i(c, b)=0$ for any $c\in \sigma \setminus \{ a\}$. 
It follows that if $a$ is an h-curve in $S$, then so is $\phi(a)$. 
We thus have the inclusion $\phi(V_{bp}(S))\subset V_{bp}(S)$ and obtain assertion (iii) by Proposition \ref{prop-max}.

Let $b$ be a separating curve in $S$ which is not an h-curve in $S$. 
Choose h-curves $b_1,\ldots, b_g$ in $S$ such that $\{ b, b_1,\ldots, b_g\}$ is a $g$-simplex of $\calc_s(S)$. 
If $\phi(b)$ were a BP in $S$, then this would contradict the fact that each of $\phi(b_1),\ldots, \phi(b_g)$ is an h-curve in $S$. 
We thus have $\phi(b)\in V_s(S)$. 
The inclusion $\phi(V_s(S))\subset V_s(S)$ is obtained.

When $(g, p)=(3, 0)$, any simplex $\sigma$ of $\calt(S)$ of maximal dimension consists of either three h-curves or two h-curves and one BP. 
Note that the former occurs if and only if for each $\alpha \in \sigma$, there exists a vertex $\beta$ of $\calt(S)$ with $i(\alpha, \beta)=0$ and $i(\gamma, \beta)\neq 0$ for any $\gamma \in \sigma \setminus \{ \alpha \}$. 
It follows that $\phi$ preserves h-curves in $S$. 
If we have a BP $b$ in $S$ with $\phi(b)$ an h-curve in $S$, then choose four mutually distinct h-curves $c_1$, $c_2$, $c_3$ and $c_4$ in $S$ disjoint from $b$ such that $c_1$ and $c_2$ are contained in a single component of $S_b$ and $c_3$ and $c_4$ are contained in another component of $S_b$. 
It then follows that $\phi(c_1)$, $\phi(c_2)$, $\phi(c_3)$ and $\phi(c_4)$ are h-curves in the component of $S_{\phi(b)}$ homeomorphic to $S_{2, 1}$, denoted by $Q$.
Let $R$ be the subsurface of $Q$ filled by $\phi(c_1)$ and $\phi(c_2)$, and let $R'$ be the one filled by $\phi(c_3)$ and $\phi(c_4)$.
We then have $|\chi(R)|\geq 2$ and $|\chi(R')|\geq 2$, and $R$ and $R'$ are disjoint. 
This contradicts $|\chi(Q)|=3$.
\end{proof}

%%%%%%%%%%%%%%%%%%%%%%%%%%%%%%%%%%%%%%%%%%%%%

\subsection{The case $g=1$ and $p=3$}

We put $S=S_{1, 3}$ throughout this subsection.
We mean by a {\it pentagon} in $\calt(S)$ the full subgraph of $\calt(S)$ spanned by five vertices $v_1,\ldots, v_5$ with $i(v_j, v_{j+1})=0$ and $i(v_j, v_{j+2})\neq 0$ for each $j$ mod $5$ (see Figure \ref{fig-pen}).
\begin{figure}
\begin{center}
\includegraphics[width=12cm]{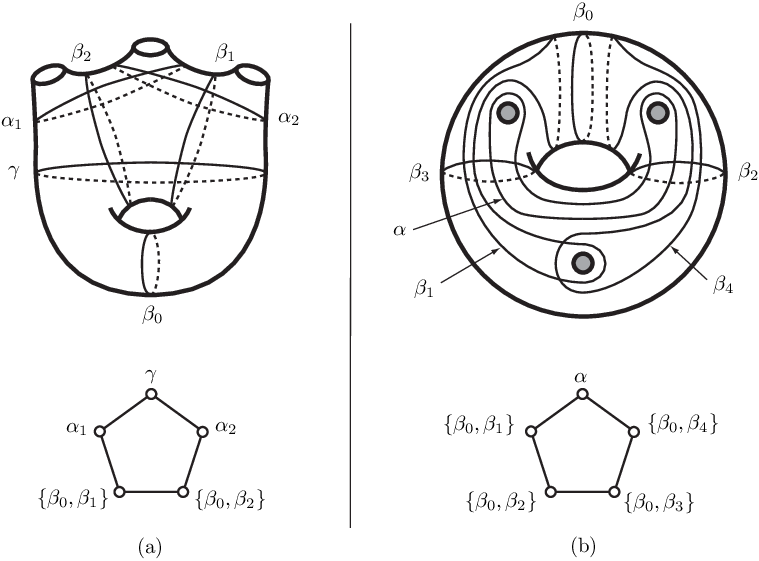}
\caption{Two pentagons in $\calt(S_{1, 3})$}\label{fig-pen}
\end{center}
\end{figure}
In this case, we say that the pentagon is defined by the 5-tuple $(v_1,\ldots, v_5)$.

\begin{lem}\label{lem-pen-bp}
There exists no pentagon in $\calt(S)$ consisting of only BP-vertices.
\end{lem}

\begin{proof}
We assume that there exists a 5-tuple $(v_1,\ldots, v_5)$ defining a pentagon in $\calt(S)$ with $v_j\in V_{bp}(S)$ for each $j$. 
Let $\partial_1$, $\partial_2$ and $\partial_3$ denote the three components of $\partial S$.
We define a map $\theta \colon V_{bp}(S)\rightarrow \{ 1, 2, 3\}$ as follows.
For each BP $b$ in $S$, the number $\theta(b)\in \{ 1, 2, 3\}$ is defined so that $\partial_{\theta(b)}$ is contained in the pair of pants cut off by $b$ from $S$.
Since distinct numbers are associated to two adjacent BP-vertices in $\calt(S)$, we may assume
\[\theta(v_1)=3,\quad \theta(v_2)=\theta(v_4)=1,\quad \theta(v_3)=\theta(v_5)=2.\]
Let $R$ denote the surface obtained by attaching a disk to $\partial_3$. 
We define $\calc^*(R)$ as the simplicial cone over $\calc(R)$ with the cone point $\ast$. 
We have the simplicial map $\pi \colon \calc(S)\rightarrow \calc^*(R)$ associated with the inclusion of $S$ into $R$, where $\pi^{-1}(\{ \ast \})$ consists of all p-curves in $S$ cutting off a pair of pants containing $\partial_3$. 
The following facts then hold:
\begin{itemize}
\item If $u, v\in V_{bp}(S)$ satisfy $i(u, v)=0$, $\theta(u)=1$ and $\theta(v)=2$, then we have $\pi(u)=\pi(v)$ and it is an edge of $\calc(R)$.
\item If $w\in V_{bp}(S)$ satisfies $\theta(w)=3$, then $\pi(w)$ consists of exactly one vertex of $\calc(R)$ corresponding to a non-separating curve in $R$. 
\item If $a, b\in V(S)$ and $w\in V_{bp}(S)$ satisfy $i(a, b)\neq 0$, $i(a, w)=i(b, w)=0$ and $\theta(w)=3$, then we have $\pi(a), \pi(b)\in V(R)$ and $i(\pi(a), \pi(b))\neq 0$.
\end{itemize}
The first fact implies the equality $\pi(v_2)=\pi(v_3)=\pi(v_4)=\pi(v_5)$. 
We put $v_1=\{ \alpha_0, \alpha_1\}$ and $v_2=\{ \alpha_0, \alpha_2\}$. 
Note that any two adjacent BP-vertices in $\calt(S)$ have a common curve. 
One of the two curves in $v_5$, say $\alpha$, belongs to $v_1$, and the other curve $\beta$ in $v_5$ intersects $v_2$ and is disjoint from $v_1$. 
Thus, $\beta$ intersects $\alpha_2$. 
The third fact implies $i(\pi(\alpha_2), \pi(\beta))\neq 0$. 
This contradicts $\pi(v_2)=\pi(v_5)$.
\end{proof}

\begin{lem}\label{lem-si-pre-p}
Let $\phi \colon \calt(S)\rightarrow \calt(S)$ be a superinjective map, and let $a$ be a p-vertex of $\calt(S)$. 
Then $\phi(a)$ is also a p-vertex of $\calt(S)$.
\end{lem}

\begin{proof}
The image of the pentagon in Figure \ref{fig-pen} (b) via $\phi$ contains four BP-vertices by Lemma \ref{lem-bp-s} (ii). 
Lemma \ref{lem-pen-bp} implies that the other vertex is not a BP-vertex. 
Since any separating curve in $S$ disjoint from a BP in $S$ is a p-curve in $S$, we obtain the lemma.  
\end{proof}

\begin{lem}\label{lem-pen-h}
Suppose that we are given a 5-tuple $(v_1,\ldots, v_5)$ of vertices of $\calt(S)$ defining a pentagon in $\calt(S)$ such that both $v_3$ and $v_4$ are BP-vertices and both $v_2$ and $v_5$ are p-vertices. 
Then $v_1$ is an h-vertex. 
\end{lem}

\begin{proof}
Note that $v_1$ is not a p-vertex since there is no p-vertex in the link of a p-vertex in $\calt(S)$. 
We number components of $\partial S$ from one to three as in the proof of Lemma \ref{lem-pen-bp}. 
For each BP or p-curve $a$ in $S$, let $Q$ be the component of $S_a$ containing exactly one component of $\partial S$. 
We associate to $a$ the number labeled the component of $\partial S$ contained in $Q$.
Distinct numbers are associated to two adjacent BP-vertices in $\calt(S)$, and the same number is associated to a p-vertex and a BP-vertex which are adjacent in $\calt(S)$. 
We may therefore assume that the number 1 is associated to $v_2$ and $v_3$ and that the number 2 is associated to $v_4$ and $v_5$. 
It follows that $v_1$ is not a BP-vertex, and thus it is an h-vertex.   
\end{proof}

Using the pentagon in Figure \ref{fig-pen} (a) and Lemma \ref{lem-pen-h}, we obtain the following:

\begin{prop}\label{prop-sep-p3}
Put $S=S_{1, 3}$ and let $\phi \colon \calt(S)\rightarrow \calt(S)$ be a superinjective map. 
Then $\phi$ preserves p-vertices and h-vertices of $\calt(S)$, respectively. 
In particular, the inclusion $\phi(V_s(S))\subset V_s(S)$ holds.
\end{prop}

%%%%%%%%%%%%%%%%%%%%%%%%%%%%%%%%%%%%%%%%%%%%%

\subsection{The case $g=1$ and $p\geq 4$}

The aim of this subsection is to prove that any superinjective map from $\calt(S)$ into itself preserves vertices corresponding to separating curves in the case of $S=S_{1, p}$ with $p\geq 4$. 
We first introduce rooted simplices of $\calt(S)$ to prove it.

\begin{defn}
Let $S$ be a surface and let $\sigma$ be a simplex of $\calt(S)$ consisting of BP-vertices. 
We say that $\sigma$ is {\it rooted} if any two BPs in $\sigma$ are BP-equivalent and there exists a non-separating curve $\alpha$ in $S$ contained in any BP of $\sigma$. 
If $|\sigma|\geq 2$, then $\alpha$ is uniquely determined and called the {\it root curve} for $\sigma$.
\end{defn}

\begin{lem}\label{lem-rooted}
Let $S=S_{g, p}$ be a surface with $g\geq 1$, and suppose either $g+p\geq 4$ or $(g, p)= (3, 0)$. 
Let $\phi \colon \calt(S)\rightarrow \calt(S)$ be a superinjective map, and let $\sigma$ be a simplex of $\calt(S)$ consisting of BP-vertices. 
If $\sigma$ is rooted, then so is $\phi(\sigma)$.
\end{lem}

\begin{proof}
We first prove that $\phi$ preserves rooted simplices consisting of $g+p-2$ BP-vertices. 
Let $\sigma =\{ b_1,\ldots, b_n\}$ be such a simplex, where we put $n=g+p-2$. 
It then follows that for each $j$, there exists a vertex $a_j\in V_t(S)$ such that $i(a_j, b_j)\neq 0$ and $i(a_j, b_k)=0$ for any $k$ distinct from $j$. 
This implies that for each $j$, there exists a non-separating curve $c_j$ contained in $\phi(b_j)$, but not in $\phi(b_k)$ for any $k$ distinct from $j$. 
Let $c_0$ be the curve in $\phi(b_1)$ that is not equal to $c_1$. 
Note that $c_0, c_1,\ldots, c_n$ are pairwise distinct curves and BP-equivalent to each other by Lemma \ref{lem-bp-s} (iii). 
Proposition \ref{prop-max} implies that there exist at most $n+1$ non-separating curves in $S$ such that each BP of $\phi(\sigma)$ consists of two of them. 
We therefore obtain the equality $\phi(b_j)=\{ c_0, c_j\}$ for each $j$. 
It follows that $\phi(\sigma)$ is rooted.

The lemma now follows because a simplex of $\calt(S)$ consisting of BP-vertices is rooted if and only if it is contained in a rooted simplex consisting of $g+p-2$ BP-vertices.
\end{proof}

\begin{lem}\label{lem-simp-ind}
Let $S$ be a surface in Lemma \ref{lem-rooted}, and let $\phi \colon \calt(S)\rightarrow \calt(S)$ be a superinjective map. 
Suppose that we are given two distinct and disjoint BPs $b_1=\{ \alpha_0, \alpha_1\}$ and $b_2=\{ \alpha_0, \alpha_2\}$ in $S$ with the common curve $\alpha_0$. 
We put $\phi(b_j)=\{ \beta_0, \beta_j\}$ for each $j=1, 2$. 
Then we have $\phi(\{ \alpha_1, \alpha_2\})=\{ \beta_1, \beta_2\}$.
\end{lem}

\begin{proof}
Note that Lemma \ref{lem-rooted} implies that $\phi(b_1)$ and $\phi(b_2)$ have the common curve $\beta_0$. 
If the conclusion of the lemma were not true, then $\phi(\{ \alpha_1, \alpha_2\})$ would be a BP containing $\beta_0$ by Lemma \ref{lem-rooted} since for each $j=1, 2$, $b_j$ and $\{ \alpha_1, \alpha_2\}$ form a rooted 1-simplex. 
In general, the maximal dimension of rooted simplices in any simplex of $\calt(S)$ of maximal dimension is equal to $g+p-3$. 
By choosing a simplex of $\calt(S)$ of maximal dimension containing $b_1$ and $b_2$ and by using injectivity of $\phi$, we can deduce a contradiction.
\end{proof}

\begin{lem}\label{lem-p-curve}
Let $S=S_{1, p}$ be a surface with $p\geq 4$, and let $\phi \colon \calt(S)\rightarrow \calt(S)$ be a superinjective map. 
If $a$ is a p-curve in $S$, then we have $\phi(a)\in V_s(S)$ and it is a p-curve in $S$.
\end{lem}

\begin{proof}
Let $a$ be a p-curve in $S$ and choose a rooted simplex $\sigma$ consisting of $p-2$ BPs in $S$ disjoint from $a$. 
Once it is shown that $\phi(a)$ belongs to $V_s(S)$, it follows that $\phi(a)$ is a p-curve in $S$ because any separating curve in $S$ disjoint from the rooted simplex $\phi(\sigma)$ is a p-curve in $S$.

We assume $\phi(a)\in V_{bp}(S)$. 
It then follows that $\phi(a)$ and each BP in $\phi(\sigma)$ are BP-equivalent since the genus of $S$ is equal to one. 
Note that $p-1$ non-separating curves in $S$ appear in BPs of $\phi(\sigma)$ and that at most $p$ non-separating curves in $S$ appear in BPs of any simplex of $\calt(S)$. 
There thus exists a BP $b$ in $\sigma$ such that $\phi(a)$ and $\phi(b)$ have a common curve.

Assume that this common curve is not equal to the root curve for $\phi(\sigma)$. 
Let $b'$ be a BP of $\sigma$ distinct from $b$. 
We then have $\phi(b)\subset \phi(a)\cup \phi(b')$. 
This contradicts the existence of a vertex $c$ of $\calt(S)$ with $i(c, b)\neq 0$ and $i(c, a)=i(c, b')=0$. 
The common curve of $\phi(a)$ and $\phi(b)$ is therefore the root curve for $\phi(\sigma)$.

Let $\alpha_0, \alpha_1,\ldots, \alpha_{p-2}$ be all non-separating curves in BPs of $\sigma$. 
By Lemma \ref{lem-simp-ind}, $\phi$ induces the map $\phi_{\sigma}$ from the set of all curves of BPs in $\sigma$ into the set of all curves of BPs in $\phi(\sigma)$ such that $\phi(\{ \alpha_j, \alpha_k\})=\{ \phi_{\sigma}(\alpha_j), \phi_{\sigma}(\alpha_k)\}$ for any distinct $j$, $k$. 
Since for each $j$, the family $\{\, \{ \alpha_j, \alpha_k\} \mid k\in \{ 0, 1,\ldots, p-2\}\setminus \{ j\}\,\}$ is a rooted simplex disjoint from $a$, the argument of the previous paragraph shows that $\phi(a)$ contains $\phi_{\sigma}(\alpha_j)$ for each $j$. 
This is a contradiction.
\end{proof}

\begin{prop}
Let $S=S_{1, p}$ be a surface with $p\geq 3$, and let $\phi \colon \calt(S)\rightarrow \calt(S)$ be a superinjective map. 
Then $\phi(V_s(S))\subset V_s(S)$.
\end{prop}

\begin{proof}
By Lemma \ref{lem-p-curve}, it suffices to show that if $a$ is a separating curve in $S$ which is not a p-curve in $S$, then $\phi(a)$ belongs to $V_s(S)$. 
We prove this claim by induction on $p$. 
When $p=3$, this is proved in Proposition \ref{prop-sep-p3}. 
Assume $p\geq 4$ and let $a$ be a separating curve in $S$ which is not a p-curve in $S$. 
We can then find a p-curve $\alpha$ in $S$ disjoint from $a$. 
The map $\phi$ induces a superinjective map $\phi_{\alpha}\colon \calt(Q)\rightarrow \calt(R)$, where $Q$ and $R$ denote the components of $S_{\alpha}$ and $S_{\phi(\alpha)}$, respectively, that are not a pair of pants. 
Since $Q$ and $R$ are homeomorphic and since the number of boundary components of $Q$ is less than that of $S$, the hypothesis of the induction implies that we have $\phi_{\alpha}(V_s(Q))\subset V_s(R)$. 
Since $a$ is an element of $V_s(Q)$, $\phi(a)$ is an element of $V_s(R)$.
We thus conclude that $\phi(a)$ belongs to $V_s(S)$.
\end{proof}

%%%%%%%%%%%%%%%%%%%%%%%%%%%%%%%%%%%%%%%%%%%%%

\subsection{Superinjective maps from the complex of separating curves into itself}

We have proved that any superinjective map $\phi$ from $\calt(S)$ into itself preserves $V_s(S)$. 
It follows that $\phi$ induces a superinjective map from $\calc_s(S)$ into itself. 
In this subsection, we deal with general superinjective maps from $\calc_s(S)$ into itself.

Let $\sigma$ be a simplex of $\calc_s(S)$ of maximal dimension.
We say that two curves $a$, $b$ in $\sigma$ are {\it adjacent with respect to} $\sigma$ if there exists a component of $S_{\sigma}$ containing $a$ and $b$ as boundary components. 
We define the {\it adjacency graph} $\cal{G}(\sigma)$ for $\sigma$ as the simplicial graph consisting of vertices in $\sigma$ and edges corresponding to adjacency with respect to $\sigma$. 
Adjacency graphs for simplices of $\calc(S)$ of maximal dimension are introduced by Irmak \cite{irmak1} in the work on superinjective maps from $\calc(S)$ into itself.

\begin{lem}\label{lem-ad}
Let $S=S_{g, p}$ be a surface with $|\chi(S)|\geq 4$, and let $\phi \colon \calc_s(S)\rightarrow \calc_s(S)$ be a superinjective map. 
Then for any simplex $\sigma$ of $\calc_s(S)$ of maximal dimension, $\phi$ induces an isomorphism between the adjacency graphs of $\sigma$ and $\phi(\sigma)$.
\end{lem}

\begin{proof}
We follow the proof of Lemma 5.1 in \cite{be-m}. 
It suffices to show that $\phi$ preserves adjacency and non-adjacency with respect to $\sigma$. 
We claim that two curves in $\sigma$ are adjacent with respect to $\sigma$ if and only if there exists a separating curve in $S$ which intersects both of them and is disjoint from any other curve of $\sigma$.
The ``if" part is clear.
Assume that $\alpha$ and $\beta$ are curves in $\sigma$ adjacent with respect to $\sigma$.
Pick a curve $\gamma$ in $S$ with $i(\alpha, \gamma)\neq 0$, $i(\beta, \gamma)\neq 0$ and $i(\delta, \gamma)=0$ for any $\delta \in \sigma \setminus \{ \alpha, \beta\}$.
If $\gamma$ is non-separating in $S$, then $t_{\gamma}(\alpha)$ is a separating curve in $S$ and satisfies $i(\alpha, t_{\gamma}(\alpha))\neq 0$, $i(\beta, t_{\gamma}(\alpha))\neq 0$ and $i(\delta, t_{\gamma}(\alpha))=0$ for any $\delta \in \sigma \setminus \{ \alpha, \beta\}$.
The ``only if" part of the claim thus follows.
The claim implies that $\phi$ preserves adjacency.

Let $\sigma =\{ a_1, b_1, c_1,\ldots, c_n\}$, where $n=2g+p-5$. 
If $a_1$ and $b_1$ are not adjacent with respect to $\sigma$, then one can find two separating curves $a_2$, $b_2$ in $S$ such that $\{ a_j, b_k, c_1, \ldots, c_n\}$ is a simplex of $\calc_s(S)$ of maximal dimension for any $j, k\in \{ 1, 2\}$; and the four vertices $a_1$, $b_1$, $a_2$ and $b_2$ form a square in $\calc_s(S)$ in this order.

Conversely, if $a_1$ and $b_1$ are adjacent with respect to $\sigma$, then there exists a component $Q$ of $S_c$ containing $a_1$ and $b_1$, where we put $c=\{ c_1,\ldots, c_n\}$. 
We then have $|\chi(Q)|=3$. 
Since there exists no square in $\calc_s(Q)$, this proves that $\phi$ preserves non-adjacency.
\end{proof}

\begin{lem}\label{lem-chi}
Let $S=S_{g, p}$ be a surface with $|\chi(S)|\geq 2$ and let $\phi \colon \calc_s(S)\rightarrow \calc_s(S)$ be a superinjective map. 
Then the following assertions hold:
\begin{enumerate}
\item Pick $a\in V_s(S)$.
Let $Q_1$ and $Q_2$ denote the two components of $S_a$, and let $R_1$ and $R_2$ denote the two components of $S_{\phi(a)}$. 
Then after exchanging the indices if necessary, we have $\phi(V_s(Q_j))\subset V_s(R_j)$ for each $j=1, 2$.  
\item $\phi$ is $\chi$-preserving, i.e., $\chi(Q_j)=\chi(R_j)$ for each $j=1, 2$ in assertion {\rm (i)}.
\end{enumerate}
\end{lem}

\begin{proof}
The lemma in the case $|\chi(S)|\leq 3$ is obvious. 
Suppose $|\chi(S)|\geq 4$. 
Recall that we refer as an {\it hp-curve} in $S$ a curve in $S$ which is either an h-curve or a p-curve in $S$. 
We first show that if $a$ is an hp-curve in $S$, then so is $\phi(a)$. 
Let $Q$ be the component of $S_a$ that is neither a handle nor a pair of pants. 
Superinjectivity of $\phi$ implies the inclusion $\phi(V_s(Q))\subset V_s(R)$ for some component $R$ of $S_{\phi(a)}$. 
By Lemma \ref{lem-sep-number}, we have $|\chi(R)|=|\chi(Q)|$, and thus $\phi(a)$ is an hp-curve in $S$.

Let $a$ be a separating curve in $S$ which is not an hp-curve in $S$. 
Let $Q_1$ and $Q_2$ denote the two components of $S_a$, and let $R_1$ and $R_2$ denote the two components of $S_{\phi(a)}$. 
Superinjectivity of $\phi$ implies that for each $j$, there exists $k$ with $\phi(V_s(Q_j))\subset V_s(R_k)$. 
If assertion (i) were not true, then we would have $\phi(V_s(Q_1)\cup V_s(Q_2))\subset V_s(R_j)$ for some $j$. 
By Lemma \ref{lem-sep-number}, we have $|\chi(R_j)|=|\chi(S)|-1$, and thus $\phi(a)$ is an hp-curve in $S$.

Note that any simplex of $\calc_s(S)$ contains at most $g+\lfloor p/2\rfloor$ hp-curves in $S$. 
If the number of components of $\partial S$ contained in one of components of $S_a$ is even, then by choosing a simplex of $\calc_s(S)$ containing $a$ and $g+\lfloor p/2\rfloor$ hp-curves in $S$, we can deduce a contradiction.

Suppose that the numbers of components of $\partial S$ contained in both components of $S_a$ are odd. 
Choose a simplex $\sigma$ of $\calc_s(S)$ of maximal dimension containing the two curves $a_1$ and $a_2$ described in Figure \ref{fig-si} (a).
\begin{figure}
\begin{center}
\includegraphics[width=12cm]{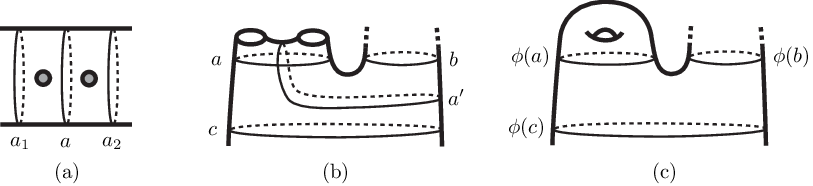}
\caption{}\label{fig-si}
\end{center}
\end{figure}
In the adjacency graph $\cal{G}(\sigma)$, $a$ and $a_j$ are adjacent for each $j=1, 2$, and $a_1$ and $a_2$ are not adjacent. 
Let $P$ denote the pair of pants in $S_{\phi(\sigma)}$ that contains the hp-curve $\phi(a)$ and is distinct from the one cut off by $\phi(a)$ from $S$. 
Since $\phi(a)$ and $\phi(a_j)$ are adjacent for each $j=1, 2$ in the graph $\cal{G}(\phi(\sigma))$ by Lemma \ref{lem-ad}, $\phi(a_1)$ and $\phi(a_2)$ are boundary components of $P$. 
This is a contradiction because $\phi(a_1)$ and $\phi(a_2)$ are not adjacent in $\cal{G}(\phi(\sigma))$ by the same lemma.

We thus proved assertion (i). 
Assertion (ii) follows from Lemma \ref{lem-sep-number}.
\end{proof}

\begin{lem}\label{lem-top-pre-s}
Let $S=S_{g, p}$ be a surface with $|\chi(S)|\geq 4$. 
Then any superinjective map $\phi \colon \calc_s(S)\rightarrow \calc_s(S)$ preserves the topological types of vertices of $\calc_s(S)$, that is, $Q_j$ and $R_j$ are homeomorphic for each $j=1, 2$ in the notation in Lemma \ref{lem-chi}.
\end{lem}

\begin{proof}
Since the maximal number of disjoint and distinct hp-curves in $S$ is equal to $g+\lfloor p/2\rfloor$, once we show that $\phi$ preserves p-curves in $S$, then $\phi$ preserves h-curves in $S$ and consequently $\phi$ is shown to preserve the topological types of vertices of $\calc_s(S)$.
To prove it, we may assume $g\geq 1$.

When $p=0$ or $1$, the claim immediately follows because there is no p-curve in $S$. 
In what follows, we assume $p\geq 2$ and that there exists a p-curve $a$ in $S$ such that $\phi(a)$ is an h-curve in $S$. 
We can find an hp-curve $b$ in $S$ distinct and disjoint from $a$. 
Lemma \ref{lem-chi} implies that $\phi(b)$ is also an hp-curve in $S$. 
Choose $c, a'\in V_s(S)$ as in Figure \ref{fig-si} (b). 
Since any separating curve in $S_{1, 2}$ is an h-curve, $\phi(a')$ is an h-curve in $S$ (see Figure \ref{fig-si} (c)). 
This contradicts the fact that $\phi$ is $\chi$-preserving.
It follows that $\phi$ preserves p-curves in $S$.
\end{proof}

Summarizing the argument in this section, we obtain the following:

\begin{prop}\label{prop-top-pre}
Let $S=S_{g, p}$ be a surface satisfying either $g=1$ and $p=3$; or $g\geq 1$ and $|\chi(S)|\geq 4$. 
Let $\phi \colon \calt(S)\rightarrow \calt(S)$ be a superinjective map. 
Then the inclusions $\phi(V_{bp}(S))\subset V_{bp}(S)$ and $\phi(V_s(S))\subset V_s(S)$ hold, and the restriction of $\phi$ to $\calc_s(S)$ preserves the topological types of vertices of $\calc_s(S)$.
\end{prop}

The latter assertion in the proposition means that for each $a\in V_s(S)$, if $Q_1$ and $Q_2$ denote the two components of $S_a$ and if $R_1$ and $R_2$ denote the two components of $S_{\phi(a)}$, then after exchanging the indices if necessary, for each $j=1, 2$, $Q_j$ and $R_j$ are homeomorphic and the inclusion $\phi(V_s(Q_j))\subset V_s(R_j)$ holds.

%%%%%%%%%%%%%%%%%%%%%%%%%%%%%%%%%%%%%%%%%%%%%

%%%%%%%%%%%%%%%%%%%%%%%%%%%%%%%%%%%%%%%%%%%%%

\section{Superinjective maps from the Torelli complex into itself}\label{sec-super-torelli}

Let $S$ be a surface of genus one. 
Given a superinjective map $\phi \colon \calt(S)\rightarrow \calt(S)$, we construct a simplicial map $\Phi \colon \calc(S)\rightarrow \calc(S)$ inducing $\phi$. 
This map $\Phi$ will be defined as follows: Pick $\alpha \in V(S)$. 
If $\alpha$ is separating in $S$, then we put $\Phi(\alpha)=\phi(\alpha)$. 
If $\alpha$ is non-separating in $S$, then we choose two BPs $a$, $b$ in $S$ such that the pair $\{ a, b\}$ is a rooted 1-simplex of $\calt(S)$ whose root curve is equal to $\alpha$. 
We then define $\Phi(\alpha)$ as the root curve for the rooted 1-simplex $\{ \phi(a), \phi(b)\}$ of $\calt(S)$. 
Sections \ref{subsec-si-p3} and \ref{subsec-si-p4} are devoted to showing that our construction of $\Phi$ is well-defined. 
In Section \ref{subsec-simp-inj}, we prove that $\Phi$ is an injective simplicial map. 
We then conclude that $\Phi$ is an automorphism of $\calc(S)$ by using Theorem \ref{thm-sha}.

\subsection{The case $g=1$ and $p=3$}\label{subsec-si-p3}

We put $S=S_{1, 3}$ throughout this subsection.

\begin{lem}\label{lem-pent-share}
Suppose that we are given a pentagon in $\calt(S)$ whose vertices are labeled as in Figure \ref{fig-pent} (a).
\begin{figure}
\begin{center}
\includegraphics[width=12cm]{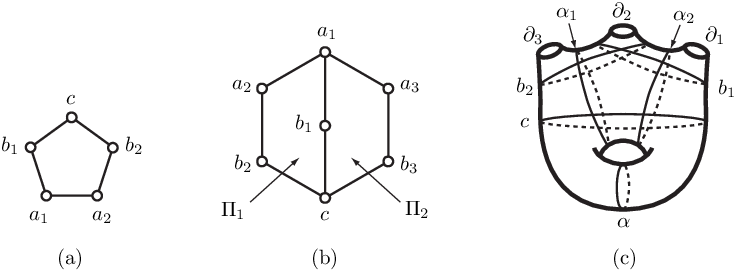}
\caption{}\label{fig-pent}
\end{center}
\end{figure}
Assume that $a_1$ and $a_2$ are BP-vertices, $b_1$ and $b_2$ are p-vertices and $c$ is an h-vertex. 
Let $\alpha$ denote the root curve for the rooted simplex $\{ a_1, a_2\}$. 
Then we have $i(\alpha, b_1)=i(\alpha, b_2)=i(\alpha, c)=0$.
\end{lem}

\begin{proof}
Since $a_j$ and $b_j$ are disjoint for each $j=1, 2$, we have $i(\alpha, b_1)=i(\alpha, b_2)=0$. 
Note that the subsurface of $S$ filled by the two p-curves $b_1$ and $b_2$ is homeomorphic to $S_{0, 4}$ and contains $\partial S$. 
The h-curve $c$ is then a boundary curve of this subsurface. 
Since $\alpha$ is disjoint from $b_1$ and $b_2$, it is also disjoint from $c$. 
\end{proof}

\begin{lem}\label{lem-pen-root}
Let $\Pi_1$ and $\Pi_2$ be pentagons in $\calt(S)$ sharing two edges which share a vertex as described in Figure \ref{fig-pent} (b). 
Assume that $a_1$, $a_2$ and $a_3$ are BP-vertices, $b_1$, $b_2$ and $b_3$ are p-vertices and $c$ is an h-vertex. 
Then the root curves for the simplices $\{ a_1, a_2\}$ and $\{ a_1, a_3\}$ are equal.
\end{lem}

\begin{proof}
Let $\alpha$ denote the root curve for $\{ a_1, a_2\}$. 
If the root curve for $\{ a_1, a_3\}$ were not equal to $\alpha$, then we could put
\[a_1=\{ \alpha, \alpha_1\},\quad a_2=\{ \alpha, \alpha_2\},\quad a_3=\{ \alpha_1, \alpha_3\}\]
with $\alpha_3\neq \alpha$. 
Applying Lemma \ref{lem-pent-share} to $\Pi_1$, we have $i(\alpha, c)=0$. Since $i(a_1, c)\neq 0$, we have $i(\alpha_1, c)\neq 0$. 
On the other hand, applying Lemma \ref{lem-pent-share} to $\Pi_2$, we obtain $i(\alpha_1, c)=0$. 
This is a contradiction. 
\end{proof}

\begin{lem}\label{lem-pen-chain}
Let $\alpha$ be a non-separating curve in $S$, and let $a_1$, $a_2$ and $a_3$ be BPs in $S$ such that each of the pairs $\{ a_1, a_2\}$ and $\{ a_1, a_3\}$ is a rooted simplex of $\calt(S)$ whose root curve is $\alpha$. 
Then there exists a sequence of pentagons in $\calt(S)$, $\Pi_1,\ldots, \Pi_n$, satisfying the following three conditions:
\begin{enumerate}
\item[(a)] For each $j$, the pentagon $\Pi_j$ consists of h-, p-, BP-, BP- and p-vertices in this order. 
\item[(b)] $a_2\in \Pi_1$, $a_3\in \Pi_n$ and $a_1\in \Pi_j$ for each $j$.
\item[(c)] For each $j$, $\Pi_j$ and $\Pi_{j+1}$ have two common edges which share either a BP-vertex or a p-vertex. 
\end{enumerate}  
\end{lem}

\begin{proof}
We put $a_j=\{ \alpha, \alpha_j\}$ for $j=1, 2, 3$. 
Choose two p-curves $b_1$, $b_2$ and an h-curve $c$ in $S$ as described in Figure \ref{fig-pent} (c). 
We then have the pentagon $\Pi_1$ in $\calt(S)$ consisting of the vertices $a_1$, $a_2$, $b_2$, $c$ and $b_1$ in this order. 
Label components of $\partial S$ as $\partial_1$, $\partial_2$ and $\partial_3$ as in Figure \ref{fig-pent} (c). 
Let $R$ be the component of $S_{a_1}$ that is not a pair of pants. 
It then follows that $\alpha_3$ is an element of $V(R)$.

Let $h\in \mod(S)$ be the half twist about $b_1$ exchanging $\partial_1$ and $\partial_2$ and being the identity on the component of $S_{b_1}$ that is not a pair of pants.
Let $x\in \mod(S)$ be the BP twist about $a_2$. 
We denote by $\Gamma$ the subgroup of $\mod(S)$ generated by $h$ and $x$. 
Since $\Gamma$ fixes $\alpha$ and $\alpha_1$, we obtain the natural homomorphism $q\colon \Gamma \rightarrow \mod(R)$. 
We denote by $\mod(R; \alpha, \alpha_1)$ the subgroup of $\mod(R)$ consisting of all elements that fix the two components of $\partial R$ corresponding to $\alpha$ and $\alpha_1$. 
We claim that $q(\Gamma)$ is equal to $\mod(R; \alpha, \alpha_1)$. 
The element $q(h)$ is the half twist about $b_1\in V(R)$, and $q(x)$ is the Dehn twist (or its inverse) about $\alpha_2\in V(R)$. 
Hence, $q(\Gamma)<\mod(R; \alpha, \alpha_1)$. 
Since the Dehn twists about $b_1$ and $\alpha_2$ generate $\pmod(R)$ and since $q(h)$ exchanges $\partial_1$ and $\partial_2$, we have the inclusion $\mod(R; \alpha, \alpha_1)<q(\Gamma)$. 
The claim is proved. 

When we regard $\alpha_2$ and $\alpha_3$ as elements of $V(R)$, we see that $\alpha_2$ and $\alpha_3$ lie in the same orbit for the action of $\mod(R; \alpha, \alpha_1)$ on $V(R)$ because $\alpha$ and $\alpha_1$ are contained in distinct components of $R_{\alpha_2}$ and the same holds for $R_{\alpha_3}$. 
The claim in the previous paragraph shows that there exist $t_1,\ldots, t_k\in \{ h^{\pm 1}, x^{\pm 1}\}$ with $\alpha_3=t_1\cdots t_k(\alpha_2)$.
The sequence of pentagons in $\calt(S)$,
\[\Pi_1,\ t_1(\Pi_1),\ t_1t_2(\Pi_1),\ldots,\ t_1\cdots t_k(\Pi_1),\]
then satisfies conditions (a), (b) and (c) in the lemma.
\end{proof}

Let $R$ be a surface of genus zero, and let $\partial_1$ and $\partial_2$ be distinct components of $\partial R$. 
We say that a curve $a$ in $R$ {\it separates} $\partial_1$ and $\partial_2$ if $\partial_1$ and $\partial_2$ are contained in distinct components of $R_a$.

\begin{prop}\label{prop-d-conn}
Let $R$ be a surface homeomorphic to $S_{0, p}$ with $p\geq 5$, and choose two distinct components $\partial_1$, $\partial_2$ of $\partial R$. 
Then the full subcomplex $\mathcal{D}=\mathcal{D}(R; \partial_1, \partial_2)$ of $\calc(R)$ spanned by all vertices corresponding to curves in $R$ which separate $\partial_1$ and $\partial_2$ is connected.
\end{prop}

\begin{proof}
We follow the idea in Lemma 2.1 of \cite{putman-conn} proving connectivity of a simplicial complex on which $\pmod(R)$ acts.
Let $\partial_1,\ldots, \partial_p$ denote components of $\partial R$ and put $J=\{ 1,\ldots, p\}$.
\begin{figure}
\begin{center}
\includegraphics[width=8cm]{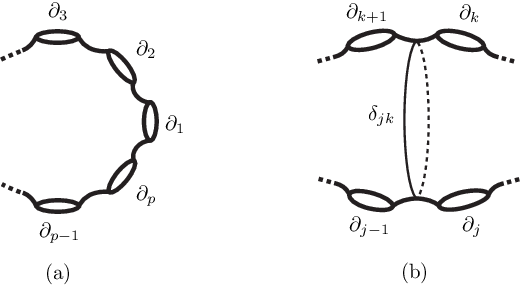}
\caption{(b) If $k=p$, then $\partial_{k+1}$ is replaced with $\partial_1$.}\label{fig-braid}
\end{center}
\end{figure}
When components of $\partial R$ are denoted as in Figure \ref{fig-braid} (a), it is known that the family of Dehn twists $t_{jk}$ about the simple closed curve $\delta_{jk}$ described in Figure \ref{fig-braid} (b) for any two integers $j, k\in J$ with $2\leq j<k\leq p$ and $1\leq k-j\leq p-3$ generates the pure mapping class group $\pmod(R)$ (see Chapters 1 and 4 in \cite{birman}).
Let us denote by $N\subset J^2$ the set of all pairs $(j, k)\in J^2$ satisfying these two inequalities. 
We put $\alpha_0=\delta_{23}$.

Given a curve $a$ in $R$ and a decomposition $J=J_1\sqcup J_2$ with $|J_1|, |J_2|\geq 2$, let us say that $a$ {\it decomposes} $J$ into $J_1$ and $J_2$ if one component of $R_a$ contains $\partial_j$ for each $j\in J_1$ and another component of $R_a$ contains $\partial_k$ for each $k\in J_2$.

\begin{claim}
Let $J=J_1\sqcup J_2$ be a decomposition of $J$ into two subsets such that $1\in J_1$, $2\in J_2$ and both $J_1$ and $J_2$ contain at least two elements. 
Then one can find a path in $\mathcal{D}$ connecting $\alpha_0$ and a vertex $\alpha$ of $\mathcal{D}$ which decomposes $J$ into $J_1$ and $J_2$.
\end{claim}

\begin{proof}
If $3\in J_2$, then one can readily find a curve $\alpha$ such that $i(\alpha, \alpha_0)=0$ and $\alpha$ decomposes $J$ into $J_1$ and $J_2$. 
Assume $3\in J_1$. 
If $|J_1|\geq 3$, then one can find a path of vertices in $\cal{D}$, $\alpha_0$, $\beta$, $\alpha$, such that $\beta$ decomposes $J$ into $J_1\setminus \{ 3\}$ and $J_2\cup \{ 3\}$ and $\alpha$ decomposes $J$ into $J_1$ and $J_2$. 
If $|J_1|=2$, then $J_1=\{ 1, 3\}$. 
One can then find a path $\alpha_0$, $\alpha_1$, $\alpha_2$, $\alpha_3$ in $\cal{D}$ consisting of p-vertices and satisfying the following: For each $j=1, 2, 3$, let $P_j$ denote the pair of pants cut off by $\alpha_j$ from $R$. 
Then $P_1$ contains $\partial_1$ and $\partial_4$, $P_2$ contains $\partial_2$ and $\partial_5$, and $P_3$ contains $\partial_1$ and $\partial_3$.
\end{proof}

\begin{claim}
For each $(j, k)\in N$, there exists a path in $\mathcal{D}$ connecting $\alpha_0$ and $t_{jk}(\alpha_0)$.
\end{claim}

\begin{proof}
Pick $(j, k)\in N$. If $j\neq 3$, then $i(\alpha_0, \delta_{jk})=0$, and thus $t_{jk}(\alpha_0)=\alpha_0$. 
If $j=3$ and $k\leq p-1$, then $i(\delta_{3k}, \delta_{2, p-1})=0$, and thus $i(t_{3k}(\alpha_0), \delta_{2, p-1})=0$. 
Since $\delta_{2, p-1}$ is a vertex of $\mathcal{D}$ and since $i(\alpha_0, \delta_{2, p-1})=0$, one can connect $\alpha_0$ and $t_{3k}(\alpha_0)$.
\begin{figure}
\begin{center}
\includegraphics[width=12cm]{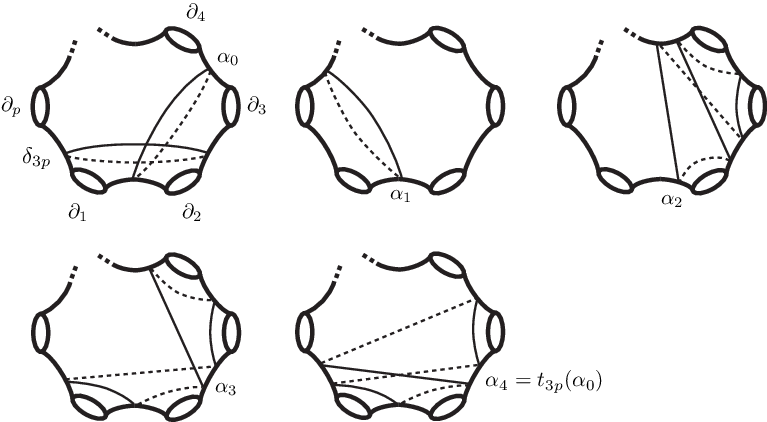}
\caption{}\label{fig-conn}
\end{center}
\end{figure}
Figure \ref{fig-conn} shows that $\alpha_0$ and $t_{3p}(\alpha_0)$ can be connected in $\mathcal{D}$ via $\alpha_1$, $\alpha_2$ and $\alpha_3$.
\end{proof}

The second claim implies that each point of the orbit for the action of $\pmod(R)$ on $\mathcal{D}$ containing $\alpha_0$ can be connected with $\alpha_0$.
The first claim then implies that $\mathcal{D}$ is connected. 
\end{proof}

\begin{lem}\label{lem-bp-a-conn}
Let $S=S_{1, p}$ be a surface with $p\geq 3$, and let $\alpha$ be a non-separating curve in $S$. 
Then the full subcomplex of $\calt(S)$ spanned by all vertices corresponding to BPs in $S$ containing $\alpha$ is connected.
\end{lem}

\begin{proof}
Let $R$ be the surface obtained by cutting $S$ along $\alpha$. 
We denote by $\partial_1$ and $\partial_2$ the two components of $\partial R$ corresponding to $\alpha$. 
There is a natural one-to-one correspondence between vertices of the complex $\mathcal{D}=\mathcal{D}(R; \partial_1, \partial_2)$ and BP-vertices of $\calt(S)$ containing $\alpha$. 
Proposition \ref{prop-d-conn} then shows the lemma. 
\end{proof}

Let $\phi \colon \calt(S)\rightarrow \calt(S)$ be a superinjective map. 
We define a map $\Phi \colon V(S)\rightarrow V(S)$ as follows: Pick $\alpha \in V(S)$. 
If $\alpha$ is separating in $S$, then we put $\Phi(\alpha)=\phi(\alpha)$. 
If $\alpha$ is non-separating in $S$, then we choose two BPs $a$, $b$ in $S$ such that the pair $\{ a, b\}$ is a rooted 1-simplex of $\calt(S)$ whose root curve is equal to $\alpha$. 
We then define $\Phi(\alpha)$ as the root curve for the rooted 1-simplex $\{ \phi(a), \phi(b)\}$ of $\calt(S)$. 
Lemma \ref{lem-rooted} implies that the pair $\{ \phi(a), \phi(b)\}$ is rooted. 
By Lemmas \ref{lem-pen-chain} and \ref{lem-bp-a-conn}, any two rooted 1-simplices $\{ a_1, b_1\}$, $\{ a_2, b_2\}$ of $\calt(S)$ with the same root curve $\alpha$ can be connected by a sequence of pentagons in $\calt(S)$ such that any two successive pentagons in it have two common edges which share either a BP-vertex or a p-vertex. 
Lemma \ref{lem-pen-root} then shows that the root curves for $\{ \phi(a_1), \phi(b_1)\}$ and $\{ \phi(a_2), \phi(b_2)\}$ are equal. 
It follows that $\Phi$ is well-defined.

%%%%%%%%%%%%%%%%%%%%%%%%%%%%%%%%%%%%%%%%%%%%%

\subsection{The case $g=1$ and $p\geq 4$}\label{subsec-si-p4}

We put $S=S_{1, p}$ with $p\geq 4$ throughout this subsection.

\begin{lem}\label{lem-bp-2-conn}
Let $a$, $b$ and $c$ be BPs in $S$ such that both of the pairs $\{ a, b\}$ and $\{ b, c\}$ are rooted 1-simplices of $\calt(S)$ and the root curves for them are equal. 
Let $\alpha$ denote the root curve. 
Then there exists a sequence $a_1, a_2,\ldots, a_n$ of BPs in $S$ such that for each $j=0, 1,\ldots, n$, $a_j$ contains $\alpha$ and $\{ b, a_j, a_{j+1}\}$ is a 2-simplex of $\calt(S)$ with $a_0=a$ and $a_{n+1}=c$.  
\end{lem}

\begin{proof}
Let $R$ be the surface obtained by cutting $S$ along $\alpha$. 
We denote by $\partial_1$ and $\partial_2$ the two components of $\partial R$ corresponding to $\alpha$. 
The surface $R$ is homeomorphic to $S_{0, p+2}$, and we have $p+2\geq 6$. 
There is a natural one-to-one correspondence between vertices of the complex $\mathcal{D}=\mathcal{D}(R; \partial_1, \partial_2)$ in Proposition \ref{prop-d-conn} and BP-vertices of $\calt(S)$ containing $\alpha$. 
It therefore suffices to prove that the link of each vertex of $\mathcal{D}$ is connected. 
Let $\beta$ be a vertex of $\mathcal{D}$. 
If $\beta$ is not a p-curve in $R$, then it is clear that the link of $\beta$ in $\mathcal{D}$ is connected. 
If $\beta$ is a p-curve in $R$, then the component of $R_{\beta}$ that is not a pair of pants is homeomorphic to $S_{0, p+1}$. 
Since we have $p+1\geq 5$, Proposition \ref{prop-d-conn} implies that the link of $\beta$ in $\mathcal{D}$ is connected.
\end{proof}

Let $\phi \colon \calt(S)\rightarrow \calt(S)$ be a superinjective map. 
We define a map $\Phi \colon V(S)\rightarrow V(S)$ in the same manner as in the previous subsection. 
Namely, we define $\Phi =\phi$ on $V_s(S)$ and if $\alpha$ is a non-separating curve in $S$, then we choose two BPs $a$, $b$ in $S$ such that the pair $\{ a, b\}$ is a rooted 1-simplex of $\calt(S)$ whose root curve is equal to $\alpha$. 
We then define $\Phi(\alpha)$ as the root curve for the rooted 1-simplex $\{ \phi(a), \phi(b)\}$ of $\calt(S)$. 
Using Lemmas \ref{lem-bp-a-conn} and \ref{lem-bp-2-conn}, we can find a sequence of rooted 2-simplices between any two given rooted 1-simplices with the same root curve $\alpha$. 
Since $\phi$ preserves rooted simplices of $\calt(S)$, the map $\Phi$ is well-defined.

%%%%%%%%%%%%%%%%%%%%%%%%%%%%%%%%%%%%%%%%%%%%

\subsection{Simplicity and injectivity}\label{subsec-simp-inj}

In this subsection, we fix a surface $S=S_{1, p}$ with $p\geq 3$ and fix a superinjective map $\phi \colon \calt(S)\rightarrow \calt(S)$. 
Let $\Phi \colon V(S)\rightarrow V(S)$ be the map constructed in Sections \ref{subsec-si-p3} and \ref{subsec-si-p4}. 
We first prove that $\Phi$ is in fact a map inducing $\phi$.

\begin{lem}\label{lem-phi-bp}
The equality $\phi(\{ \alpha, \beta \})= \{ \Phi(\alpha), \Phi(\beta)\}$ holds for each BP $\{ \alpha, \beta \}$ in $S$.
\end{lem}

\begin{proof}
By the definition of $\Phi$, both $\Phi(\alpha)$ and $\Phi(\beta)$ are contained in $\phi(\{ \alpha, \beta \})$.  
After choosing a non-separating curve $\gamma$ in $S$ such that both $\{ \beta, \gamma \}$ and $\{ \gamma, \alpha \}$ are BPs in $S$, we conclude that $\Phi(\alpha)$, $\Phi(\beta)$ and $\Phi(\gamma)$ are mutually distinct by Lemma \ref{lem-simp-ind}.
\end{proof}

\begin{lem}
The map $\Phi$ defines a simplicial map from $\calc(S)$ into itself.
\end{lem}

\begin{proof}
Let $\alpha$ and $\beta$ be two distinct curves in $S$ with $i(\alpha, \beta)=0$. 
If both $\alpha$ and $\beta$ are separating in $S$, then $i(\Phi(\alpha), \Phi(\beta))=0$ because we have $\Phi =\phi$ on $V_s(S)$ and $\phi$ is simplicial. 
If both $\alpha$ and $\beta$ are non-separating in $S$, then $\{ \alpha, \beta \}$ is a BP in $S$ because $S$ is of genus one. 
Lemma \ref{lem-phi-bp} shows that $\{ \Phi(\alpha), \Phi(\beta)\}$ is also a BP in $S$, and thus $i(\Phi(\alpha), \Phi(\beta))=0$.

Suppose that $\alpha$ is non-separating in $S$ and $\beta$ is separating in $S$.
Note that $\alpha$ is in the component of $S_{\beta}$ of genus one.
Unless $\beta$ is an h-curve in $S$, we can find a non-separating curve $\gamma$ in $S$ such that $i(\beta, \gamma)=0$ and the pair $\{ \alpha, \gamma \}$ is a BP in $S$.
It follows from $i(\{ \alpha, \gamma \}, \beta)=0$ that we have $i(\phi(\{ \alpha, \gamma \}), \phi(\beta))=0$ and thus $i(\Phi(\alpha), \Phi(\beta))=0$.

We now assume that $\beta$ is an h-curve in $S$.
Choose separating curves $\gamma_1$, $\gamma_2$ in $S$ and non-separating curves $\alpha_1$, $\alpha_2$ in $S$ as described in Figure \ref{fig-simp}.
\begin{figure}
\begin{center}
\includegraphics[width=7cm]{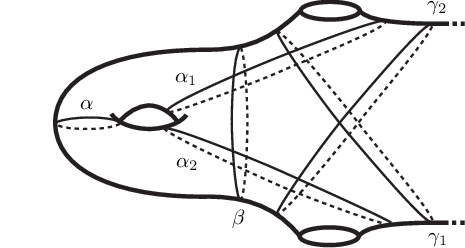}
\caption{}\label{fig-simp}
\end{center}
\end{figure}
Note that $\beta$ is the only h-curve in $S$ disjoint from both of $\gamma_1$ and $\gamma_2$. 
By Proposition \ref{prop-top-pre}, the same property holds for the image of these separating curves via $\phi$. 
Since $\phi(\{ \alpha, \alpha_1\})=\{ \Phi(\alpha), \Phi(\alpha_1)\}$ and $\phi(\gamma_2)$ are disjoint and since $\phi(\{ \alpha, \alpha_2\})=\{ \Phi(\alpha), \Phi(\alpha_2)\}$ and $\phi(\gamma_1)$ are disjoint, the curve $\Phi(\alpha)$ is disjoint from $\phi(\gamma_1)$ and $\phi(\gamma_2)$. 
It follows that $\Phi(\alpha)$ is disjoint from $\Phi(\beta)$.
\end{proof}

\begin{lem}\label{lem-Phi-inj}
The simplicial map $\Phi \colon \calc(S)\rightarrow \calc(S)$ is injective.
\end{lem}

\begin{proof}
Since $\Phi$ preserves separating curves and non-separating curves, respectively, and since the restriction of $\Phi$ to $V_s(S)$ is superinjective, it is enough to show that for any two non-separating curves $\alpha$, $\beta$ in $S$, the equality $\Phi(\alpha)=\Phi(\beta)$ implies $\alpha =\beta$. 
Note that each curve in $S_{\alpha}$ is either separating in $S$ or BP-equivalent to $\alpha$ in $S$. 
The map $\Phi$ induces a superinjective map $\Phi_{\alpha}\colon \calc(S_{\alpha})\rightarrow \calc(S_{\Phi(\alpha)})$ because $\phi$ is superinjective. 
Since both $S_{\alpha}$ and $S_{\Phi(\alpha)}$ are homeomorphic to $S_{0, p+2}$, Theorem \ref{thm-sha} implies that $\Phi_{\alpha}$ is an isomorphism. 
We also obtain an isomorphism $\Phi_{\beta}\colon \calc(S_{\beta})\rightarrow \calc(S_{\Phi(\beta)})$. If $\Phi(\alpha)=\Phi(\beta)$, then we have the equality
\begin{align*}
\phi(V_s(S)\cap V(S_{\alpha}))&=\Phi(V_s(S)\cap V(S_{\alpha}))=V_s(S)\cap V(S_{\Phi(\alpha)})\\
&=V_s(S)\cap V(S_{\Phi(\beta)})=\Phi(V_s(S)\cap V(S_{\beta}))=\phi(V_s(S)\cap V(S_{\beta})).
\end{align*}
Since $\phi$ is injective, the equality $V_s(S)\cap V(S_{\alpha})=V_s(S)\cap V(S_{\beta})$ holds. 
We thus have $\alpha =\beta$.
\end{proof}

Applying Theorem \ref{thm-sha}, we obtain the following conclusion.

\begin{thm}\label{thm-g-1-si}
Let $S=S_{1, p}$ be a surface with $p\geq 3$, and let $\phi \colon \calt(S)\rightarrow \calt(S)$ be a superinjective map. 
Then there exists an automorphism $\Phi$ of $\calc(S)$ such that we have $\Phi(\alpha)=\phi(\alpha)$ for any separating curve $\alpha$ in $S$ and $\{ \Phi(\beta), \Phi(\gamma)\} =\phi(\{ \beta, \gamma \})$ for any BP $\{ \beta, \gamma \}$ in $S$. 
\end{thm}

\begin{rem}
The construction of the simplicial map $\Phi \colon \calc(S)\rightarrow \calc(S)$ associated with a superinjective map $\phi \colon \calt(S)\rightarrow \calt(S)$ is valid for a surface $S=S_{g, p}$ with $g+p\geq 5$ as well after establishing connectivity of an appropriate complex as in Proposition \ref{prop-d-conn}. 
However, when $g\geq 2$, one cannot show injectivity of $\Phi$ along the proof of Lemma \ref{lem-Phi-inj}.
This is because for each non-separating curve $\alpha$ in $S$, the full subcomplex of $\calt(S)$ spanned by all vertices corresponding to either a separating curve in $S$ disjoint from $\alpha$ or a BP in $S$ containing $\alpha$ can be identified with a {\it proper} subcomplex of $\calc(S_{\alpha})$.  
If $\phi$ is an automorphism of $\calt(S)$, then one can conclude that $\Phi$ is an automorphism of $\calc(S)$ by using $\phi^{-1}$.

In Theorem \ref{thm-g-2}, we will prove that if $S=S_{g, p}$ is a surface with $g\geq 2$ and $|\chi(S)|\geq 4$, then any automorphism of $\calt(S)$ is induced by an automorphism of $\calc(S)$, after proving the same conclusion for any automorphism of $\calc_s(S)$.
\end{rem}

%%%%%%%%%%%%%%%%%%%%%%%%%%%%%%%%%%%%%%%%%%%%%

\section{Automorphisms of the complex of separating curves}\label{sec-aut-sep}

Let $S=S_{g, p}$ be a surface in Theorems \ref{thm-tor-comm} and \ref{thm-jo-comm}. 
Given an automorphism $\phi$ of $\calc_s(S)$, we construct an automorphism $\Phi$ of $\calc(S)$ which extends $\phi$. 
When $g=1$, we first focus on the case $p=3$. 
Results of this case are used in the case $p\geq 4$ in an inductive argument on $p$. 
When $g\geq 2$, the construction of $\Phi$ follows the argument, due to Brendle-Margalit \cite{bm}, using sharing pairs and spines in $S$.

\subsection{The case $g=1$ and $p=3$}\label{subsec-hex}

We put $S=S_{1, 3}$ throughout this subsection. 
Note that the link of each h-vertex in $\calc_s(S)$ consists of only p-vertices and that the link of each p-vertex in $\calc_s(S)$ consists of only h-vertices. 
It follows that there exists no pentagon in $\calc_s(S)$. 
We mean by a {\it hexagon} in $\calc_s(S)$ the full subgraph of $\calc_s(S)$ spanned by six vertices $v_1,\ldots, v_6$ with $i(v_j, v_{j+1})=0$, $i(v_j, v_{j+2})\neq 0$ and $i(v_j, v_{j+3})\neq 0$ for each $j$ mod $6$ (see Figure \ref{fig-hex}).
\begin{figure}
\begin{center}
\includegraphics[width=12cm]{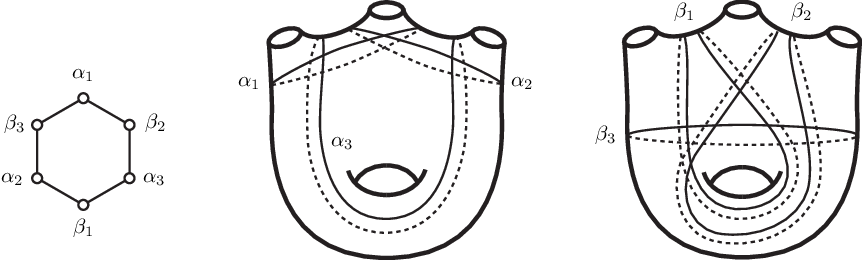}
\caption{A hexagon in $\calc_s(S_{1, 3})$}\label{fig-hex}
\end{center}
\end{figure}
In this case, we say that the hexagon is defined by the 6-tuple $(v_1,\ldots, v_6)$.

To prove the next lemma, we need the following terminology.
We say that a simple arc $l$ in a surface $R$ with $\partial R\neq \emptyset$ is said to be {\it essential} in $R$ if
\begin{itemize}
\item $\partial l$ consists of two distinct points of $\partial R$;
\item $l$ meets $\partial R$ only at its end points; and
\item $l$ is not isotopic relative to $\partial l$ to an arc in $\partial R$.
\end{itemize}
For a simple closed curve $C$ in $S$, we denote by $S_C$ the surface obtained by cutting $S$ along $C$.
If $C$ is separating in $S$, then each component of $S_C$ is naturally identified with a subsurface of $S$.

\begin{lem}\label{lem-hex}
Suppose that we have a 6-tuple $(\alpha_1, \beta_3, \alpha_2, \beta_1, \alpha_3, \beta_2)$ of vertices of $\calc_s(S)$ defining a hexagon in $\calc_s(S)$ with $\alpha_j$ a p-vertex for each $j=1, 2, 3$.
For each $j=1, 2, 3$, choose representatives $A_j$ and $B_j$ of $\alpha_j$ and $\beta_j$, respectively, such that any two of the six curves intersect minimally.
Then the following assertions hold:
\begin{enumerate}
\item For each $j=1, 2, 3$, let $Q_j$ denote the component of $S_{A_j}$ homeomorphic to $S_{1, 2}$, and let $\partial_j$ denote the component of $\partial S$ contained in $Q_j$. 
Then $\partial_1$, $\partial_2$ and $\partial_3$ are mutually distinct.
\item For each $j=1, 2, 3$, let $R_j$ denote the component of $S_{B_j}$ homeomorphic to $S_{0, 4}$. 
Then for any distinct $j, k=1, 2, 3$, the intersection $B_j\cap R_k$ consists of essential simple arcs in $R_k$ which are mutually isotopic, where isotopy of essential simple arcs in $R_k$ may move their end points, keeping them staying in $\partial R_k$.
Moreover, if $l\in \{ 1, 2, 3\}$ is the number distinct from $j$ and $k$, then each component of $B_j\cap R_k$ cuts off an annulus containing $\partial_l$ from $R_k$.
\item For each $k$ mod $3$, a component of $B_{k+1}\cap R_k$ and a component of $B_{k+2}\cap R_k$ can be isotoped so that they are disjoint. 
\end{enumerate}
\end{lem}

\begin{proof}
We first prove assertion (i). 
Suppose that assertion (i) is not true. 
We may assume $\partial_1=\partial_2$, and denote it by $\partial$. 
Let $\bar{S}$ be the surface obtained by attaching a disk to $\partial$, and let $\pi \colon \calc(S)\rightarrow \calc^*(\bar{S})$ be the simplicial map associated with the inclusion of $S$ into $\bar{S}$, where $\calc^*(\bar{S})$ is the simplicial cone over $\calc(\bar{S})$ with the cone point $\ast$. 
Note that $\pi^{-1}(\{ \ast \})$ consists of all p-curves in $S$ cutting off a pair of pants containing $\partial$. 
If $\alpha$ is a p-curve in $S$ such that $\partial$ is contained in the component of $S_{\alpha}$ homeomorphic to $S_{1, 2}$ and if $\beta$ is an h-curve in $S$ disjoint from $\alpha$, then $\pi(\alpha)=\pi(\beta)$. 
This implies that $\pi(\beta_2)=\pi(\alpha_1)=\pi(\beta_3)=\pi(\alpha_2)=\pi(\beta_1)$. 
If $\partial_3\neq \partial$, then the equality $\pi(\beta_2)=\pi(\beta_1)$ implies $\beta_2=\beta_1$, and this is a contradiction. 
We therefore have $\partial_3=\partial$ and $\pi(\alpha_3)$ is equal to the image via $\pi$ of the other five vertices. 
This contradicts the fact that $\pi^{-1}(\{ \gamma \})$ is a tree for each $\gamma \in V(\bar{S})$ (see Theorem 7.1 of \cite{kls}).
Assertion (i) is proved.

The minimality of $|B_2\cap B_3|$ implies that $B_2\cap R_3$ consists of essential simple arcs in $R_3$.
Since $A_1$ cuts off a pair of pants containing $\partial_2$ and $\partial_3$ from $R_3$, the intersection $B_2\cap R_3$ consists of mutually isotopic, essential simple arcs in $R_3$ each of which cuts off an annulus containing $\partial_1$, assertion (ii) follows when $j=2$, $k=3$ and $l=1$.
The other cases can be proved similarly.

Suppose that assertion (iii) is not true for $k=3$. 
The other cases are discussed similarly. 
As noted in the previous paragraph, $B_2\cap R_3$ consists of mutually isotopic, essential simple arcs in $R_3$ each of which cuts off an annulus containing $\partial_1$. 
Similarly, $B_1\cap R_3$ consists of mutually isotopic, essential simple arcs in $R_3$ each of which cuts off an annulus containing $\partial_2$.
Choose a component $l_2$ of $B_2\cap R_3$ and a component $l_1$ of $B_1\cap R_3$. 
Since $l_1$ and $l_2$ can not be isotoped so that they are disjoint, there exist a subarc of $l_1$ and a subarc of $l_2$ such that the union of them is a simple closed curve in $R_3$ isotopic to $\partial_1$. 
This contradicts the fact that $A_3$ is a boundary component of a regular neighborhood of $B_1\cup B_2$ and is a p-curve cutting off a pair of pants containing $\partial_1$. 
Assertion (iii) follows. 
\end{proof}

Let $R$ be a surface such that $\partial R$ consists of at least two components.
Given two distinct components $\partial$, $\partial'$ of $\partial R$, we say that an essential simple arc $l$ in $R$ {\it connects $\partial$ and $\partial'$} if one of the end points of $l$ lies in $\partial$ and another in $\partial'$.

We note that there is a one-to-one correspondence between the isotopy classes of p-curves in $R$ and of essential simple arcs in $R$ connecting two distinct components of $\partial R$, where isotopy of essential simple arcs in $R$ may move the end points of arcs, keeping them staying in $\partial R$.
In fact, one associates to a p-curve $C$ in $R$ an essential simple arc in $R$ disjoint from $C$ and connecting the two components of $\partial R$ contained in the pair of pants cut off by $C$ from $R$.
Conversely, for each essential simple arc $l$ in $R$ connecting two distinct components $\partial$, $\partial'$ of $\partial R$, the p-curve in $R$ corresponding to $l$ is obtained as a boundary component of a regular neighborhood of the union $l\cup \partial \cup \partial'$ in $R$.

\begin{thm}\label{thm-hex}
The action of $\pmod(S)$ on the set of hexagons in $\calc_s(S)$ is transitive.
\end{thm}

\begin{proof}
Let $(\alpha_1, \beta_3, \alpha_2, \beta_1, \alpha_3, \beta_2)$ be a 6-tuple of vertices of $\calc_s(S)$ defining a hexagon in $\calc_s(S)$ with $\alpha_j$ a p-vertex for each $j=1, 2, 3$.
For each $j=1, 2, 3$, we choose representatives $A_j$ and $B_j$ of $\alpha_j$ and $\beta_j$, respectively, such that any two of the six curves intersect minimally.
For each $j=1, 2, 3$, let $l_j$ be an essential simple arc in $S$ corresponding to the p-curve $A_j$.
By Lemma \ref{lem-hex}, for any distinct $j, k=1, 2, 3$, $l_j$ and $l_k$ can be isotoped so that they are disjoint.

For each $j=1, 2, 3$, let $\partial_j$ denote the component of $\partial S$ contained in the component of $S_{A_j}$ homeomorphic to $S_{1, 2}$.
We denote by $A(S)$ the set of isotopy classes of essential simple arcs in $S$ and denote by $[r]\in A(S)$ the isotopy class of an essential simple arc $r$ in $S$.
We note that $\pmod(S)$ acts transitively on the set of triplets $([r_1], [r_2], [r_3])$ of elements in $A(S)$ such that
\begin{itemize}
\item for each $j$ mod $3$, $r_j$ connects $\partial_{j+1}$ and $\partial_{j+2}$; and
\item $r_1$, $r_2$ and $r_3$ are pairwise disjoint.
\end{itemize}
The theorem then follows because $[l_1]$, $[l_2]$ and $[l_3]$ determine $\alpha_1$, $\alpha_2$ and $\alpha_3$ and because $B_j$ is obtained as a boundary component of a regular neighborhood of the union $A_{j+1}\cup A_{j+2}$ in $S$ for each $j$ mod $3$.  
\end{proof}

For each hexagon $\Pi$ in $\calc_s(S)$, we can find a non-separating curve in $S$ disjoint from any curve corresponding to a vertex of $\Pi$ by using Figure \ref{fig-hex} and Theorem \ref{thm-hex}.
Such a curve in $S$ is unique up to isotopy because for each essential simple arc $l$ in a handle $H$, an essential simple closed curve in $H$ disjoint from $l$ uniquely exists up to isotopy.
We denote by $c(\Pi)\in V(S)$ that non-separating curve in $S$.

\begin{lem}\label{lem-equal-c}
If $\Pi_1$ and $\Pi_2$ are hexagons in $\calc_s(S)$ sharing two edges which share a p-vertex, then we have $c(\Pi_1)=c(\Pi_2)$.
\end{lem}

\begin{proof}
Let $\beta_1$ and $\beta_2$ be two h-vertices of a hexagon $\Pi$ in $\calc_s(S)$, and let $\alpha$ be the p-vertex of $\Pi$ adjacent to both of $\beta_1$ and $\beta_2$.
The lemma follows from the fact that $c(\Pi)$ is the only non-separating curve in $S$ disjoint from any of $\alpha$, $\beta_1$ and $\beta_2$.  
\end{proof}

\begin{lem}\label{lem-hex-chain}
Let $c$ be a non-separating curve in $S$, and let $\Pi$ and $\Delta$ be hexagons in $\calc_s(S)$ with $c(\Pi)=c(\Delta)=c$. 
Then there exists a sequence of hexagons in $\calc_s(S)$, $\Pi_1,\ldots, \Pi_n$, satisfying the following two conditions:
\begin{enumerate}
\item[(a)] $\Pi_1=\Pi$ and $\Pi_n=\Delta$.
\item[(b)] For each $j$, $\Pi_j$ and $\Pi_{j+1}$ share two edges which share a p-vertex. 
\end{enumerate}
\end{lem}

To prove this lemma, we need the following proposition, which can be proved along the same idea as in the proof of Proposition \ref{prop-d-conn}.

\begin{prop}\label{prop-d-conn2}
We put $R=S_{0, 5}$ and choose two components $\partial_1$, $\partial_2$ of $\partial R$. 
We define $\cal{E}$ as the simplicial graph so that
\begin{itemize}
\item vertices are elements of $V(R)$ corresponding to a curve in $R$ cutting off a pair of pants containing $\partial_1$ and $\partial_2$; and
\item two such vertices $\alpha$, $\beta$ are connected by an edge if and only if $i(\alpha, \beta)=4$.
\end{itemize}
Then the graph $\cal{E}$ is connected.
\end{prop}

\begin{proof}
Let $\partial_1,\ldots, \partial_5$ denote the boundary components of $R$.
For two integers $j$, $k$ with $2\leq j<k\leq 5$ and $1\leq k-j\leq 2$, let $\delta_{jk}\in V(R)$ be the curve described in Figure \ref{fig-braid} (b).
As noted in the proof of Proposition \ref{prop-d-conn}, $\pmod(R)$ is generated by all $t_{\delta_{jk}}$.
One can check that $\delta_{35}$ is a vertex of $\cal{E}$ and that for any other $\delta_{jk}$, we have either $t_{\delta_{jk}}(\delta_{35})=\delta_{35}$ or $i(t_{\delta_{jk}}(\delta_{35}), \delta_{35})=4$.
Since $\pmod(R)$ acts transitively on the set of vertices of $\cal{E}$, connectivity of $\cal{E}$ follows.
\end{proof}

\begin{proof}[Proof of Lemma \ref{lem-hex-chain}]
We first prove the lemma when $\Pi$ and $\Delta$ have a common h-vertex. 
Let $(\alpha_1, \beta_3, \alpha_2, \beta_1, \alpha_3, \beta_2)$ be a 6-tuple of vertices of $\calc_s(S)$ defining $\Pi$ with $\alpha_j$ a p-vertex for each $j=1, 2, 3$, as described in Figure \ref{fig-hex}. 
Assume that $\Delta$ contains $\beta_3$. 
If we denote by $h$ the involution that is described in Figure \ref{fig-inv} (a) and exchanges $\alpha_1$ and $\alpha_2$, then we have $h(\Pi)=\Pi$.
\begin{figure}
\begin{center}
\includegraphics[width=12cm]{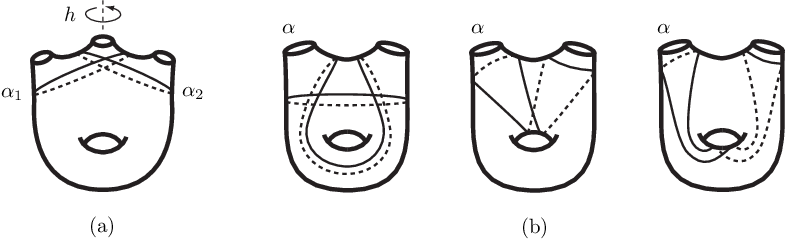}
\caption{(b) Four h-curves in $S$ any two of which are connected by an edge in the graph $\cal{G}(\alpha)$}\label{fig-inv}
\end{center}
\end{figure}
For $j=1, 2$, we denote by $P_j$ the pair of pants cut off by $\alpha_j$ from $S$ and define $h_j$ as the half twist about $\alpha_j$ with $h_j^2=t_{\alpha_j}$ that exchanges the two components of $\partial S$ contained in $P_j$ and is the identity on the complement of $P_j$ in $S$.
Let $H$ be the handle cut off by $\beta_3$ from $S$.
One can check that the product $hh_1h_2h_1$ fixes $\beta_3$, $\alpha_1$ and $\alpha_2$ and that its restriction to $H$ is the involution in the center of $\mod(H)$.
Moreover, we have $t_{\beta_3}=(hh_1h_2h_1)^2$.
It follows that $h$, $h_1$, $h_2$ and $t_c$ generate the stabilizer of $\beta_3$ and $c$ in $\mod(S)$. 
Theorem \ref{thm-hex} implies that $\Pi$ is sent to $\Delta$ by a product of elements in $\{ h, h_1^{\pm 1}, h_2^{\pm 1}, t_c^{\pm 1}\}$. 
As in the proof of Lemma \ref{lem-pen-chain}, we can find a sequence of hexagons in $\calc_s(S)$ satisfying conditions (a) and (b).

In general, Proposition \ref{prop-d-conn2} shows that there exists a sequence of vertices in $\calc_s(S)$, $a_1, a_2,\ldots, a_n$, such that
\begin{itemize}
\item $a_j\neq a_{j+1}$ and $i(a_j, a_{j+1})=i(a_j, c)=0$ for each $j$;
\item $a_1$ is an h-vertex in $\Pi$, and $a_n$ is an h-vertex in $\Delta$; and
\item for each $j$, if $a_j$ is a p-vertex, then the three vertices $a_{j-1}$, $a_j$ and $a_{j+1}$ lie in a hexagon of $\calc_s(S)$.
\end{itemize}
By using this sequence and the fact proved in the previous paragraph, we obtain a desired sequence of hexagons in $\calc_s(S)$.
\end{proof}

\begin{lem}\label{lem-pre-h-p}
Any superinjective map from $\calc_s(S)$ into itself preserves h-vertices and p-vertices, respectively.
\end{lem}

\begin{proof}
Pick $\alpha \in V_s(S)$. 
We define a simplicial graph $\cal{G}(\alpha)$ as follows: Vertices of $\cal{G}(\alpha)$ are vertices in the link of $\alpha$ in $\calc_s(S)$. 
Two distinct vertices $\beta_1$, $\beta_2$ in $\cal{G}(\alpha)$ are connected by an edge if and only if there exists a hexagon in $\calc_s(S)$ containing $\beta_1$, $\alpha$ and $\beta_2$ in this order. 
Note that when $\alpha$ is an h-vertex (resp.\ a p-vertex), $\beta_1$ and $\beta_2$ are connected if and only if $i(\beta_1, \beta_2)=2$ (resp.\ $4$).

If $\alpha$ is an h-vertex, then it is known that $\cal{G}(\alpha)$ is isomorphic to the Farey graph (see Section 3.2 in \cite{luo}). 
If $\alpha$ is a p-vertex, then $\cal{G}(\alpha)$ contains at least four vertices such that any two of them are connected by an edge (see Figure \ref{fig-inv} (b)). 
The lemma follows because the Farey graph does not contain such a subgraph and because the link of any h-vertex (resp.\ p-vertex) in $\calc_s(S)$ consists of only p-vertices (resp.\ h-vertices).
\end{proof}

\begin{thm}\label{thm-ext-p3}
We put $S=S_{1, 3}$. 
Then for any automorphism $\phi$ of $\calc_s(S)$, there exists an automorphism $\Phi$ of $\calc(S)$ extending $\phi$.
\end{thm}

\begin{proof}
Let $\phi $ be an automorphism of $\calc_s(S)$. 
Since $\phi$ preserves hexagons in $\calc_s(S)$, we define the extension $\Phi \colon V(S)\rightarrow V(S)$ of $\phi$ by putting $\Phi(c)=c(\phi(\Pi))$ for each non-separating curve $c$ in $S$, where $\Pi$ is a hexagon in $\calc_s(S)$ with $c=c(\Pi)$. 
This is well-defined thanks to Lemmas \ref{lem-equal-c}, \ref{lem-hex-chain} and \ref{lem-pre-h-p}. 
The map from $V(S)$ into itself associated to $\phi^{-1}$ is equal to the inverse of $\Phi$.
It follows that $\Phi$ is bijective.

We next prove that $\Phi$ is simplicial. 
It is easy to see that for any two curves $c_1$, $c_2$ in $S$ with $i(c_1, c_2)=0$, we have $i(\Phi(c_1), \Phi(c_2))=0$ unless both $c_1$ and $c_2$ are non-separating in $S$. 
Let $c_1$ and $c_2$ be distinct non-separating curves in $S$. 
We note that if $c_1$ and $c_2$ are disjoint, then there exist infinitely many p-curves in $S$ disjoint from $c_1$ and $c_2$, and there is no h-curve in $S$ disjoint from $c_1$ and $c_2$. 

Assume that $c_1$ and $c_2$ intersect. 
If the subsurface of $S$ filled by $c_1$ and $c_2$ is a handle, then there exists an h-curve in $S$ disjoint from $c_1$ and $c_2$. 
We claim that otherwise, there exists at most one p-curve in $S$ disjoint from $c_1$ and $c_2$.
Assume that there exist two distinct p-curves $d_1$, $d_2$ in $S$ disjoint from $c_1$ and $c_2$.
Let $D_1$ and $D_2$ be representatives of $d_1$ and $d_2$, respectively, with $|D_1\cap D_2|=i(d_1, d_2)$, which is positive since we have $d_1\neq d_2$.
We denote by $R$ the component of $S_{D_1}$ homeomorphic to $S_{1, 2}$.
The intersection $D_2\cap R$ consists of essential simple arcs in $R$ and contains a non-separating one $l$ in $R$, i.e., an essential simple arc in $R$ whose complement in $R$ is connected, since the subsurface of $S$ filled by $c_1$ and $c_2$ is not a handle.
Two boundary components of a regular neighborhood of $D_1\cup l$ in $R$ form a BP in $S$ cutting off a pair of pants from $S$.
This is a contradiction because $c_1$ and $c_2$ intersect and are disjoint from $d_1$ and $d_2$.
 
The observations in the previous two paragraphs show that $\Phi$ preserves two disjoint non-separating curves in $S$ and that $\Phi$ is simplicial.
\end{proof}

%%%%%%%%%%%%%%%%%%%%%%%%%%%%%%%%%%%%%%%%%%%%

\subsection{The case $g=1$ and $p\geq 4$}

Let $S=S_{1, p}$ be a surface with $p\geq 4$, and fix an automorphism $\phi$ of $\calc_s(S)$. 
We define an automorphism of $\calc(S)$ extending $\phi$ by induction on $p$. 
For an integer $q$ with $2\leq q\leq p$, we refer as a $q$-{\it HBC (hole bounding curve)} in $S$ a separating curve $\alpha$ in $S$ such that the component of $S_{\alpha}$ of genus zero contains exactly $q$ components of $\partial S$.

Let $\alpha$ be a $q$-HBC in $S$ with $2\leq q\leq p-2$. 
By the hypothesis of the induction, we obtain an isomorphism $\phi_{\alpha}\colon \lk(\alpha)\rightarrow \lk(\phi(\alpha))$ extending the restriction of $\phi$ to $\lk(\alpha)\cap \calc_s(S)$, where for each $\gamma \in V(S)$, $\lk(\gamma)$ denotes the link of $\gamma$ in $\calc(S)$.

We next assume that $\alpha$ is a $(p-1)$-HBC in $S$. 
Let $Q_1$ and $Q_2$ denote the two components of $S_{\alpha}$ with $Q_1$ of genus one. 
Choosing a separating curve $\beta$ in $Q_2$, we define an isomorphism $\phi_{\alpha}\colon \lk(\alpha)\rightarrow \lk(\phi(\alpha))$ as $\phi_{\alpha}=\phi_{\beta}$ on $V(Q_1)$ and $\phi_{\alpha}=\phi$ on $V(Q_2)$. 
Note that $\beta$ is a $q$-HBC in $S$ with $2\leq q\leq p-2$ and that $V(Q_2)$ is contained in $V_s(S)$. 
This definition is independent of the choice of $\beta$ by the following:

\begin{lem}\label{lem-uni}
We put $R=S_{1, p}$ with $p\geq 2$. 
Then any two automorphisms of $\calc(R)$ that preserve $V_s(R)$ and are equal on $V_s(R)$ are equal on $V(R)$.  
\end{lem}

\begin{proof}
Let $\phi$ and $\psi$ be two such automorphisms of $\calc(R)$.
By Theorem \ref{thm-cc}, $\phi$ and $\psi$ are induced by elements $g$ and $h$ of $\mod^*(R)$, respectively.
For any non-separating curve $\alpha$ in $R$, we choose two separating curves $\beta$, $\gamma$ in $R$ disjoint from $\alpha$ and filling $S_{\alpha}$.
Since $g^{-1}h$ fixes $\beta$ and $\gamma$, it also fixes $\alpha$ because $\alpha$ is the unique curve in $R$ disjoint from $\beta$ and $\gamma$.
It follows that $g^{-1}h$ fixes any element of $V(R)$.
\end{proof}

Let $U$ be the set of all $q$-HBCs in $S$ with $2\leq q\leq p-1$. 
Lemma \ref{lem-uni} also shows that if $\alpha_1, \alpha_2\in U$ are disjoint curves, then $\phi_{\alpha_1}=\phi_{\alpha_2}$ on $\lk(\alpha_1)\cap \lk(\alpha_2)$. 
By using the following proposition, we obtain an automorphism $\Phi$ of $\calc(S)$ as an extension of $\phi_{\alpha}$ for any $\alpha \in U$.

\begin{prop}
We put $R=S_{0, p}$ with $p\geq 6$ and choose two components $\partial_1$, $\partial_2$ of $\partial R$. 
We define $\cal{F}$ as the full subcomplex of $\calc(R)$ spanned by all vertices corresponding to a curve $\alpha$ in $R$ such that one component of $R_{\alpha}$ contains both $\partial_1$ and $\partial_2$ and contains at least three components of $\partial R$. 
Then $\cal{F}$ is connected.
\end{prop}

This proposition is also verified along the same idea as in the proof of Proposition \ref{prop-d-conn}. 
Combining Theorem \ref{thm-ext-p3}, we proved the following:

\begin{thm}\label{thm-g-1-s}
Let $S=S_{1, p}$ be a surface with $p\geq 3$. 
Then for any automorphism $\phi$ of $\calc_s(S)$, there exists an automorphism $\Phi$ of $\calc(S)$ extending $\phi$.
\end{thm}

%%%%%%%%%%%%%%%%%%%%%%%%%%%%%%%%%%%%%%%%%%%%

\subsection{The case $g\geq 2$}

The idea for the construction of $\Phi$ due to Brendle-Margalit \cite{bm} is to use sharing pairs defined below.
Let $S=S_{g, p}$ be a surface with $g\geq 2$ and $|\chi(S)|\geq 3$.
For an h-curve $C$ in $S$, we denote by $H_C$ the handle cut off by $C$ from $S$, which is naturally identified with a subsurface of $S$.

\begin{defn}\label{defn-share}
Let $S=S_{g, p}$ be a surface with $g\geq 2$ and $|\chi(S)|\geq 3$.
Let $a, b\in V_s(S)$ be h-curves in $S$ and $\beta \in V(S)$ a non-separating curve in $S$. 
We say that $a$ and $b$ {\it share} $\beta$ if there exist representatives $A$, $B$ and $\frak{b}$ of $a$, $b$ and $\beta$, respectively, such that we have $|A\cap B|=i(a, b)$, $H_A\cap H_B$ is an annulus with its core curve $\frak{b}$, and $S\setminus (H_A\cup H_B)$ is connected. 
In this case, $\{ a, b\}$ is called a {\it sharing pair} for $\beta$ or a {\it sharing pair} in $S$ if $\beta$ is not specified (see Figure \ref{fig-spine} (a)).
\begin{figure}
\begin{center}
\includegraphics[width=12cm]{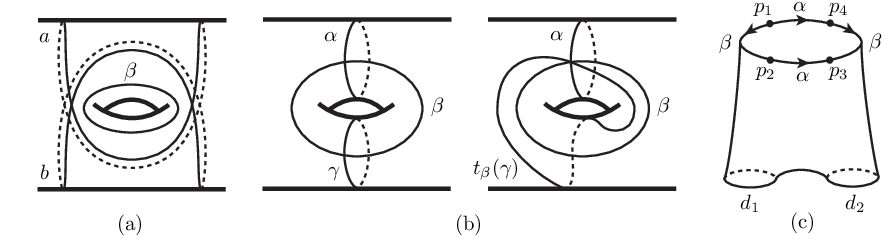}
\caption{}\label{fig-spine}
\end{center}
\end{figure}
\end{defn}

The action of $\pmod(S)$ on the set of sharing pairs in $S$ is transitive.
In fact, in the notation of Definition \ref{defn-share}, $A$ and $B$ have to intersect, and $B\cap H_A$ consists of essential simple arcs in $H_A$ which are mutually isotopic because $\frak{b}$ is a curve in $H_A$ disjoint from $B$.
Since $H_A\cap H_B$ is an annulus, $B\cap H_A$ consists of exactly two essential simple arcs $l_1$, $l_2$ in $H_A$.
We put $\partial l_j=\{ p_j, q_j\}$ for each $j=1, 2$ so that $p_1$, $q_1$, $q_2$ and $p_2$ appear along $A$ in this order.
Let $Q$ denote the complement of $H_A$ in the surface obtained by cutting $S$ along $A$.
Since $B$ is separating in $S$, $B\cap Q$ consists of two essential simple arcs $r_1$, $r_2$ in $Q$ such that $r_1$ connects $p_1$ with $p_2$ and $r_2$ connects $q_1$ with $q_2$.
Connectivity of $S\setminus (H_A\cup H_B)$ implies that $r_1$ and $r_2$ are non-separating in $Q$.
Since $B$ cuts off a handle from $S$, $r_1$ and $r_2$ are isotopic.
Transitivity of the action of $\pmod(S)$ on the set of sharing pairs in $S$ thus follows.

The claim in the previous paragraph shows that if $\{ a, b\}$ is a sharing pair for a non-separating curve $\beta$ in $S$, then we have $i(a, b)=4$, and the subsurface filled by $a$ and $b$ is homeomorphic to $S_{0, 4}$ and has two boundary components corresponding to $\beta$. 
Note that when $S$ is a surface of genus one, there exists no pair $\{ a, b\}$ of h-curves in $S$ satisfying the condition in Definition \ref{defn-share}. 
The following lemma characterizes sharing pairs in terms of disjointness and non-disjointness.

\begin{lem}\label{lem-share-cha}
Let $S=S_{g, p}$ be a surface with $g\geq 2$ and $|\chi(S)|\geq 4$, and let $a$ and $b$ be h-curves in $S$. 
Then $a$ and $b$ form a sharing pair in $S$ if and only if there exist separating curves $w$, $x$, $y$ and $z$ in $S$ satisfying the following six conditions:
\begin{itemize}
\item $z$ cuts off a surface $Q$ homeomorphic to $S_{2, 1}$ from $S$;
\item $a, b\in V(Q)$ and $i(a, b)\neq 0$;
\item $i(x, y)=0$;
\item $i(w, a)=0$, $i(w, b)=0$ and $i(w, z)\neq 0$;
\item $i(x, a)\neq 0$, $i(x, b)=0$ and $i(x, z)\neq 0$; and
\item $i(y, a)=0$, $i(y, b)\neq 0$ and $i(y, z)\neq 0$.
\end{itemize}
\end{lem}

This lemma for closed surfaces is proved in Lemma 4.1 of \cite{bm} and in Lemma 4 of \cite{bm-add}.
The same proof of the ``if" part is also valid for general surfaces.
The ``only if" part is proved by using Figure \ref{fig-share} (a) and (b) for surfaces with $g\geq 3$ and $g=2$, respectively.
The choice of the curves in Figure \ref{fig-share} (a) appears in Figure 2 of \cite{bm-add}.
\begin{figure}
\begin{center}
\includegraphics[width=12cm]{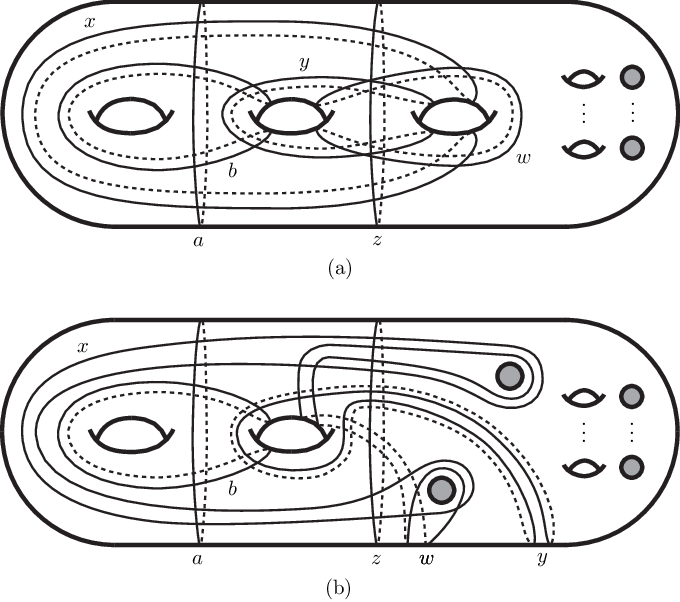}
\caption{}\label{fig-share}
\end{center}
\end{figure}

Let $\phi$ be an automorphism of $\calc_s(S)$. 
By Lemmas \ref{lem-top-pre-s} and \ref{lem-share-cha}, $\phi$ preserves sharing pairs in $S$. 
We will define an automorphism $\Phi$ of $\calc(S)$ extending $\phi$ so that for each non-separating curve $\alpha$ in $S$, $\Phi(\alpha)$ is the curve shared by $\phi(a)$ and $\phi(b)$, where $\{ a, b\}$ is a sharing pair for $\alpha$. 
Spines for sharing pairs, defined below, were introduced in \cite{bm} to show that $\Phi$ is well-defined. 
For two curves $\alpha, \beta \in V(S)$ with $i(\alpha, \beta)=1$, we denote by $H(\alpha, \beta)$ the handle filled by $\alpha$ and $\beta$.

\begin{defn}\label{defn-spine}
Let $S=S_{g, p}$ be a surface with $g\geq 2$ and $|\chi(S)|\geq 3$. 
A triplet of distinct non-separating curves in $S$, $\alpha\textrm{-}\beta\textrm{-}\gamma$, is called a {\it spine} in $S$ if the following three conditions are satisfied:
\begin{enumerate}
\item[(a)] $i(\alpha, \beta)=i(\beta, \gamma)=1$ and $i(\gamma, \alpha)\leq 1$.
\item[(b)] Let $a$ and $b$ denote the boundary components of the handles $H(\alpha, \beta)$ and $H(\beta, \gamma)$, respectively. 
Then $\{ a, b\}$ is a sharing pair for $\beta$.
\item[(c)] There exist representatives $\frak{a}$, $\frak{b}$ and $\frak{c}$ of $\alpha$, $\beta$ and $\gamma$, respectively, such that any two of them intersect minimally and $S\setminus (\frak{a} \cup \frak{b} \cup \frak{c})$ is connected. 
\end{enumerate}
In this case, $\alpha\textrm{-}\beta\textrm{-}\gamma$ is called a spine for the sharing pair $\{ a, b\}$ (see Figure \ref{fig-spine} (b)).

A {\it move} between two spines in $S$ is defined to be a change of the form, $\alpha\textrm{-}\beta\textrm{-}\gamma \mapsto \alpha\textrm{-}\beta\textrm{-}\gamma'$, with $\gamma\textrm{-}\beta\textrm{-}\gamma'$ a spine.
\end{defn}

In what follows, we describe some basic properties of spines and moves between them, which will be used to prove that $\Phi$ is well-defined.

\begin{lem}\label{lem-spine}
Let $S=S_{g, p}$ be a surface with $g\geq 2$ and $|\chi(S)|\geq 3$. 
Suppose that we are given a spine $\alpha\textrm{-}\beta\textrm{-}\gamma$ in $S$, and let $\frak{a}$, $\frak{b}$ and $\frak{c}$ be representatives of $\alpha$, $\beta$ and $\gamma$, respectively, satisfying condition (c) in Definition \ref{defn-spine}. 
Let $R$ denote the surface obtained by cutting $S$ along $\frak{a}$ and $\frak{b}$. 
Let $p_1$, $p_2$, $p_3$ and $p_4$ denote the identified points on the cut end that correspond to the single point of $\frak{a}\cap \frak{b}$ and are located as in Figure \ref{fig-spine} (c). 
Then
\begin{enumerate}
\item[(i)] if $i(\gamma, \alpha)=1$, then $|\frak{a}\cap \frak{b}\cap \frak{c}|=1$. 
In this case, $\frak{c}$ is given by an essential simple arc in $R$ connecting either $p_1$ and $p_3$ or $p_2$ and $p_4$.
\item[(ii)] if $i(\gamma, \alpha)=0$, then $\frak{c}$ is given by an essential simple arc in $R$ connecting a point in the interior of an arc corresponding to $\frak{b}$ with a point in the interior of another arc corresponding to $\frak{b}$.    
\end{enumerate}
\end{lem}

\begin{proof}
Let $Q$ denote the subsurface of $S$ filled by $\frak{a}$, $\frak{b}$ and $\frak{c}$, which contains $H(\alpha, \beta)$ and $H(\beta, \gamma)$ and is homeomorphic to $S_{1, 2}$.
Let $D_1$ and $D_2$ denote the boundary curves of $Q$.
Note that both $D_1$ and $D_2$ are essential simple closed curves in $S$.
Let $\partial$ denote the boundary component of $R$ consisting of arcs corresponding to $\frak{a}$ and $\frak{b}$.
Let $P$ denote the pair of pants cut off by $D_1\cup D_2$ from $R$ and containing $\partial$.
The curve $\frak{c}$ is given by several simple arcs in $P$ connecting two points of $\partial$.
Since $S\setminus (\frak{a}\cup \frak{b}\cup \frak{c})$ is connected, any of those arcs is an essential simple arc in $P$.
The curve $\frak{c}$ is given by a single essential simple arc $l$ in $P$ because any two essential simple arcs in $P$ connecting two points of $\partial$ are isotopic and the union of any two disjoint such arcs cuts off a disk from $P$.
The arc $l$ connects either two of $p_1$, $p_2$, $p_3$ and $p_4$ or two points in the interior of the two arcs corresponding to $\frak{b}$ because we have $|\frak{b}\cap \frak{c}|=1$. 
If $l$ connected either $p_1$ and $p_2$ or $p_3$ and $p_4$, then $\frak{c}$ could be moved into a curve disjoint from $\frak{b}$. 
This contradicts $i(\beta, \gamma)=1$.
If $l$ connected either $p_1$ and $p_4$ or $p_2$ and $p_3$, then $\frak{c}$ could be moved into a curve disjoint from $\frak{a}$.
This contradicts minimality of $|\frak{c}\cap \frak{a}|$.
The lemma thus follows.
\end{proof}

\begin{lem}\label{lem-move-sep}
Let $S=S_{g, p}$ be a surface with $g\geq 2$ and $|\chi(S)|\geq 3$. 
Let $\alpha\textrm{-}\beta\textrm{-}\gamma \mapsto \alpha\textrm{-}\beta\textrm{-}\gamma'$ be a move between two spines in $S$. 
Then the following two assertions hold:
\begin{enumerate}
\item Let $d_1$ and $d_2$ denote the boundary components of the subsurface of $S$ filled by $\alpha$, $\beta$ and $\gamma$. 
Then $\gamma'$ intersects either $d_1$ or $d_2$. 
\item If $|\chi(S)|\geq 4$, then there exists a separating curve in $S$ which intersects $\gamma'$, but is disjoint from any of $\alpha$, $\beta$ and $\gamma$.
\end{enumerate}
\end{lem}

\begin{proof}
Suppose that $\gamma'$ intersects neither $d_1$ nor $d_2$.
Let $\frak{a}$, $\frak{b}$, $\frak{c}$, $\frak{c}'$, $D_1$ and $D_2$ be representatives of $\alpha$, $\beta$, $\gamma$, $\gamma'$, $d_1$ and $d_2$, respectively, such that any two of them, except the two of $\frak{c}$ and $\frak{c}'$, intersect minimally; and $S\setminus (\frak{a}\cup \frak{b}\cup \frak{c})$ and $S\setminus (\frak{a}\cup \frak{b}\cup \frak{c}')$ are connected.
We define $R$, $p_1$, $p_2$, $p_3$ and $p_4$ as in Lemma \ref{lem-spine}.
Let $\partial$ denote the boundary component of $R$ consisting of arcs corresponding to $\frak{a}$ and $\frak{b}$, and let $P$ denote the pair of pants cut off by $D_1\cup D_2$ from $R$ and containing $\partial$.
By Lemma \ref{lem-spine}, $\frak{c}$ and $\frak{c}'$ are given by essential simple arcs $r$, $r'$ in $P$, respectively, connecting two points of $\partial$.
Each of $r$ and $r'$ connects either (1) $p_1$ and $p_3$; (2) $p_2$ and $p_4$; or (3) a point in the interior of an arc corresponding to $\frak{b}$ and a point in the interior of another such arc.
In what follows, we repeatedly use the criterion in Expos\'e 3, Proposition 10 of \cite{flp} to know the geometric intersection number of two curves in $S$.
It is impossible that the same case holds for $r$ and $r'$ because otherwise the condition $i(\gamma, \gamma')\leq 1$ would imply that $\frak{c}$ and $\frak{c}'$ are isotopic.
It is also impossible that $r$ satisfies (1) (resp.\ (2)) and $r'$ satisfies (2) (resp.\ (1)) because we have $i(\gamma, \gamma')\leq 1$.
If $r$ satisfies (1) and $r'$ satisfies (3), then the condition $i(\gamma, \gamma')\leq 1$ implies that we have the two possibilities indicated in Figure \ref{fig-pants} (a) and (b).
Let $b$ and $b'$ denote the boundary curves of the handles $H(\beta, \gamma)$ and $H(\beta, \gamma')$, respectively.
In case (a), $b$ and $b'$ are isotopic.
This contradicts the assumption that $\{ b, b'\}$ is a sharing pair in $S$.
In case (b), we have $i(b, b')=8$ (see Figure \ref{fig-pants} (c)).
\begin{figure}
\begin{center}
\includegraphics[width=12cm]{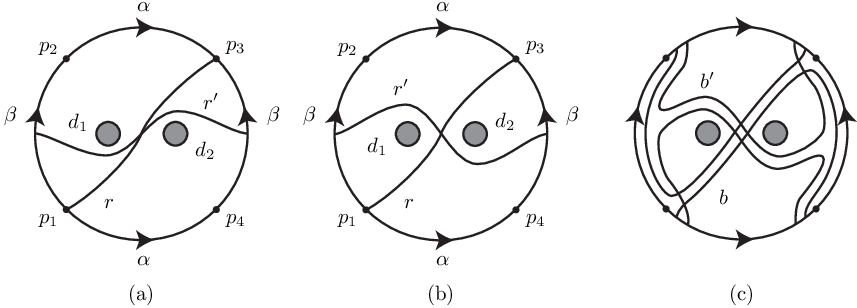}
\caption{}\label{fig-pants}
\end{center}
\end{figure}
This also contradicts the same assumption.
In the other cases, we can deduce a contradiction along verbatim argument.
Assertion (i) thus follows.
Assertion (ii) follows from assertion (i).
\end{proof}

The following lemma can be proved along the same idea as in the proof of Lemma 4.4 in \cite{bm} by using Lemma \ref{lem-move-sep}.

\begin{lem}\label{lem-move-well}
Let $S=S_{g, p}$ be a surface with $g\geq 2$ and $|\chi(S)|\geq 4$, and let $\phi \colon \calc_s(S)\rightarrow \calc_s(S)$ be a superinjective map. 
Suppose that we have a move between two spines in $S$, $\alpha\textrm{-}\beta\textrm{-}\gamma \mapsto \alpha\textrm{-}\beta\textrm{-}\gamma'$. 
Let $a$, $b$ and $b'$ denote the boundary curves of the handles $H(\alpha, \beta)$, $H(\beta, \gamma)$ and $H(\beta, \gamma')$, respectively. 
Then $\{ \phi(a), \phi(b)\}$ and $\{ \phi(a), \phi(b')\}$ are sharing pairs for the same non-separating curve in $S$.
\end{lem}

The following proposition can also be proved along the argument in the proof of Proposition 4.5 in \cite{bm}.

\begin{prop}\label{prop-move}
Let $S=S_{g, p}$ be a surface with $g\geq 2$ and $|\chi(S)|\geq 4$. 
For each non-separating curve $\beta$ in $S$, any two spines $\alpha\textrm{-}\beta\textrm{-}\gamma$ and $\delta \textrm{-}\beta\textrm{-}\epsilon$ in $S$ differ by finitely many moves.
\end{prop}

\begin{thm}\label{thm-g-2}
Let $S=S_{g, p}$ be a surface with $g\geq 2$ and $|\chi(S)|\geq 4$. Then the following assertions hold:
\begin{enumerate}
\item For any automorphism $\phi$ of $\calt(S)$, there exists an automorphism $\Phi$ of $\calc(S)$ such that we have $\Phi(\alpha)=\phi(\alpha)$ for any separating curve $\alpha$ in $S$ and $\{ \Phi(\beta), \Phi(\gamma)\} =\phi(\{ \beta, \gamma \})$ for any BP $\{ \beta, \gamma \}$ in $S$.
\item For any automorphism $\psi$ of $\calc_s(S)$, there exists an automorphism $\Psi$ of $\calc(S)$ extending $\psi$.
\end{enumerate} 
\end{thm}

\begin{proof}
Let $\psi$ be an automorphism of $\calc_s(S)$.
We define a map $\Psi \colon V(S)\rightarrow V(S)$ as follows: For each separating curve $\alpha$ in $S$, we put $\Psi(\alpha)=\psi(\alpha)$.
For each non-separating curve $\beta$ in $S$, we define $\Psi(\beta)$ as the non-separating curve shared by $\psi(a)$ and $\psi(b)$, where $\{ a, b\}$ is a sharing pair for $\beta$. 
This is well-defined thanks to Lemma \ref{lem-move-well} and Proposition \ref{prop-move}.
For an h-curve $a$ in $S$, we denote by $H_a$ the handle cut off by $a$ from $S$. 
By the definition of $\Psi$, if $a$ is an h-curve in $S$ and $\beta$ is a non-separating curve in $H_a$, then $\Psi(a)$ is also an h-curve in $S$, and $\Psi(\beta)$ is a curve in the handle $H_{\Psi(a)}$.

We next prove that $\Psi$ defines a simplicial map from $\calc(S)$ into itself. 
Let $\alpha$ and $\beta$ be disjoint and distinct curves in $S$. 
If both $\alpha$ and $\beta$ are separating in $S$, then $\Psi(\alpha)$ and $\Psi(\beta)$ are disjoint since $\psi$ is simplicial. 
If $\alpha$ is separating in $S$ and $\beta$ is non-separating in $S$, then there exists an h-curve $a$ in $S$ such that $i(a, \alpha)=0$ and $\beta$ is a curve in $H_a$. 
Since $\alpha$ is either equal to $a$ or in the complement of $H_a$, $\Psi(\alpha)$ and $\Psi(\beta)$ are disjoint.

Finally, we suppose that both $\alpha$ and $\beta$ are non-separating in $S$. 
If there exist distinct and disjoint h-curves $a$, $b$ in $S$ such that $\alpha$ lies in $H_a$ and $\beta$ lies in $H_b$, then $\Psi(\alpha)$ and $\Psi(\beta)$ are disjoint because $\Psi(\alpha)$ lies in $H_{\Psi(a)}$ and $\Psi(\beta)$ lies in $H_{\Psi(b)}$. 
Otherwise, $\alpha$ and $\beta$ form a BP in $S$.
Assuming that the BP $\{ \alpha, \beta \}$ does not cut off a pair of pants from $S$, we show that $\Psi(\alpha)$ and $\Psi(\beta)$ are disjoint.
A similar argument can also be applied in the other case.
Choose separating curves $\gamma_1$, $\gamma_2$ in $S$ whose union cuts off the surface $Q$ homeomorphic to $S_{1, 2}$ and containing $\alpha$ and $\beta$ (see Figure \ref{fig-bp}).
\begin{figure}
\begin{center}
\includegraphics[width=3.9cm]{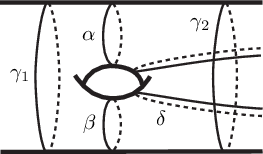}
\caption{}\label{fig-bp}
\end{center}
\end{figure}
Lemma \ref{lem-top-pre-s} implies that the union of $\Psi(\gamma_1)$ and $\Psi(\gamma_2)$ also cuts off the surface $Q'$ homeomorphic to $S_{1, 2}$ and containing $\Psi(\alpha)$ and $\Psi(\beta)$. 
Choose a separating curve $\delta$ in $S$ which intersects exactly one of $\gamma_1$ and $\gamma_2$ and is disjoint from $\alpha$ and $\beta$. 
Suppose that $\Psi(\alpha)$ and $\Psi(\beta)$ intersect. 
Let $R$ be the subsurface of $Q'$ filled by $\Psi(\alpha)$ and $\Psi(\beta)$. 
If $R$ is a handle, then let $\epsilon$ denote the boundary curve of $R$.
Since $\psi$ preserves h-curves in $S$, $\psi^{-1}(\epsilon)$ is an h-curve in $Q$.
Applying the argument until the previous paragraph to $\psi^{-1}$, we see that $\psi^{-1}(\epsilon)$ is disjoint from $\alpha$ and $\beta$.
This is a contradiction because any h-curve in $Q$ intersects either $\alpha$ or $\beta$.
It follows that $R$ is either equal to $Q'$ or a complement in $Q'$ of a tubular neighborhood of a non-separating curve in $Q'$.
This contradicts the existence of the curve $\Psi(\delta)$ which intersects exactly one of $\Psi(\gamma_1)$ and $\Psi(\gamma_2)$ and is disjoint from $\Psi(\alpha)$ and $\Psi(\beta)$.
We therefore proved that $\Psi$ defines a simplicial map from $\calc(S)$ into itself. 
Assertion (ii) is proved.

Let $\phi$ be an automorphism of $\calt(S)$.
Since $\phi$ preserves $V_s(S)$ by Lemma \ref{lem-bp-s} (i), we obtain an automorphism $\Phi$ of $\calc(S)$ extending the restriction of $\phi$ to $\calc_s(S)$ by using assertion (ii). 
Along the argument in Section 6 of \cite{bm} to find separating curves defining a BP, we can show the equality $\{ \Phi(\beta), \Phi(\gamma)\} =\phi(\{ \beta, \gamma \})$ for each BP $\{ \beta, \gamma \}$ in $S$.
\end{proof}

%%%%%%%%%%%%%%%%%%%%%%%%%%%%%%%%%%%%%%%%%%%%%

%%%%%%%%%%%%%%%%%%%%%%%%%%%%%%%%%%%%%%%%%%%%%

\section{Twisting elements of the Torelli group}\label{sec-twist}

Using the results obtained so far on automorphisms of the Torelli complex and the complex of separating curves, we compute the commensurators of the Torelli group and the Johnson kernel. 
We first present a few facts on abelian subgroups of the Torelli group and a characterization of Dehn twists about separating curves and of BP twists.

\subsection{Abelian subgroups of the Torelli group}

The argument of this subsection heavily depends on \cite{bbm}, \cite{v} and \cite{v-t}, where closed surfaces are dealt with. 
Let $S=S_{g, p}$ be a surface and let $\sigma$ be a simplex of $\calc(S)$. 
Pick a curve $\alpha$ in $\sigma$. 
We say that $\alpha$ is of {\it a-type} in $\sigma$ if $\alpha$ is separating in $S$. 
We say that $\alpha$ is of {\it b-type} in $\sigma$ if $\alpha$ is non-separating in $S$ and is contained in a BP-equivalence class in $\sigma$ consisting of at least two curves. 
Otherwise, i.e., if $\alpha$ is non-separating in $S$ and any other curve in $\sigma$ is not BP-equivalent to $\alpha$, then $\alpha$ is said to be of {\it c-type} in $\sigma$. 
Each curve of $\sigma$ is classified into these three types. 
This terminology follows that in \cite{v}. 
We denote by $D_{\sigma}$ the subgroup of $\mod(S)$ generated by Dehn twists about all curves in $\sigma$.
For a vertex $\beta$ of $\calc(S)$, we denote by $D_{\beta}$ the subgroup of $\mod(S)$ generated by $t_{\beta}$.
The following theorem relies on \cite{bbm} and \cite{v}.

\begin{thm}\label{thm-vau}
Let $S=S_{g, p}$ be a surface with $3g+p-4\geq 0$. 
Let $\sigma$ be a simplex of $\calc(S)$ if $3g+p-4>0$, and let $\sigma$ be a vertex of $\calc(S)$ otherwise. 
Then
\begin{enumerate}
\item $D_{\sigma}\cap \cali(S)$ is generated by Dehn twists about curves in $\sigma \cap V_s(S)$ and BP twists about BPs in $S$ of two curves in $\sigma$. 
\item $D_{\sigma}\cap \calk(S)$ is generated by Dehn twists about curves in $\sigma \cap V_s(S)$.
\end{enumerate}
\end{thm}

\begin{proof}
When $g=0$, the theorem follows because we have $\cali(S)=\calk(S)=\pmod(S)$ and $V_s(S)=V(S)$. 
When $(g, p)=(1, 1)$, the theorem follows because both $\cali(S)$ and $\calk(S)$ are trivial and any vertex of $\calc(S)$ corresponds to a non-separating curve in $S$. 
When $g\geq 2$ and $p=0$, assertions (i) and (ii) are proved in Theorem 3.1 of \cite{v} and Theorem A.1 of \cite{bbm}, respectively.

We prove the theorem by induction on the number of boundary components of $S$. 
We assume either $g=1$ and $p\geq 2$ or $g\geq 2$ and $p\geq 1$. 
Let $R$ denote the surface obtained by attaching a disk to a boundary component $\partial$ of $S$.
We then have the simplicial map $\pi \colon \calc(S)\rightarrow \calc^*(R)$ associated with the inclusion of $S$ into $R$, where $\calc^*(R)$ is the simplicial cone over $\calc(R)$ with the cone point $\ast$. 
Note that $\pi^{-1}(\{ \ast \})$ consists of all p-curves in $S$ cutting off a pair of pants containing $\partial$ from $S$. 
We have the natural homomorphism $q\colon \pmod(S)\rightarrow \pmod(R)$ satisfying $q(\cali(S))=\cali(R)$ and $q(\calk(S))=\calk(R)$.

Let $\alpha$ and $\beta$ be distinct and disjoint curves in $S$ which are non-separating in $S$.
The following facts then hold:
\begin{enumerate}
\item[(a)] $\alpha$ and $\beta$ are BP-equivalent in $S$ if and only if $\pi(\alpha)$ and $\pi(\beta)$ are BP-equivalent in $R$.
\item[(b)] If $\pi(\alpha)=\pi(\beta)$, then the union of $\alpha$ and $\beta$ cuts off a pair of pants containing $\partial$ from $S$.
\end{enumerate}

Let $\sigma$ be a simplex of $\calc(S)$.
To prove assertion (i), we suppose that there exists an element $x$ of $D_{\sigma}\cap \cali(S)$ which is not in the group generated by Dehn twists about curves in $\sigma \cap V_s(S)$ and BP twists about BPs in $S$ of two curves in $\sigma$.
We may assume that $x$ is a product of Dehn twists and their inverses about curves in $\sigma$ which are non-separating in $S$ and any two of which are not BP-equivalent in $S$.
We can then deduce a contradiction by using fact (a) and the hypothesis of the induction.
Assertion (i) is proved.

To prove assertion (ii), we suppose that there exists an element $y$ of $D_{\sigma}\cap \calk(S)$ which is not in the group generated by Dehn twists about curves in $\sigma \cap V_s(S)$. 
We may assume that $y$ is a product of Dehn twists and their inverses about curves in $\sigma$ which are non-separating in $S$. 
By assertion (i), $y$ is a product of BP-twists about BPs in $S$ of two curves in $\sigma$. 
By fact (b) and the hypothesis of the induction, $y$ is a non-zero power of the BP-twist about a BP in $S$ cutting off a pair of pants containing $\partial$ from $S$. 
This contradicts Proposition \ref{prop-jo-bp}.
Assertion (ii) is proved.
\end{proof}

For a finitely generated abelian group $A$, we denote by ${\rm rank}(A)$ the rank of $A$.
Note that for any surface $S$, any abelian subgroup of $\mod(S)$ is finitely generated by Theorem A in \cite{blm}.
We say that a subgroup of $\mod(S)$ is {\it reducible} if it fixes a simplex of $\calc(S)$.

\begin{lem}\label{lem-abe}
Let $S=S_{g, p}$ be a surface with $3g+p-4>0$, and let $\sigma$ be a simplex of $\calc(S)$.
We denote by $\nu=\nu(\sigma)$ the number of components of $S_{\sigma}$ and denote by $\Omega =\Omega(\sigma)$ the number of components of $S_{\sigma}$ which are neither a handle nor a pair of pants. 
Then
\begin{enumerate}
\item the inequality ${\rm rank}(D_{\sigma}\cap \cali(S))\leq \nu -1$ holds.
\item the inequality ${\rm rank}(D_{\sigma}\cap \cali(S))+\Omega \leq 2g+p-3$ holds. If the equality holds, then each component $Q$ of $S_{\sigma}$ satisfies $|\chi(Q)|\leq 2$.
\item If $A$ is an abelian reducible subgroup of $\cali(S)$ whose canonical reduction system is equal to $\sigma$, then ${\rm rank}(A)\leq {\rm rank}(D_{\sigma}\cap \cali(S))+\Omega$.
\end{enumerate}
\end{lem}

Assertions (i) and (ii) can be verified along argument in the proof of Lemmas 3.1 and 3.2 in \cite{v}, respectively.
Assertion (iii) follows from Corollary 7.18 in \cite{iva-subgr}. 
We refer to \cite{iva-subgr} for the definition of canonical reduction systems for subgroups of $\mod(S)$. 
We can show the following proposition by using Lemma \ref{lem-abe} and following the proof of Theorem 3.3 in \cite{v}.

\begin{prop}\label{prop-v}
Let $S=S_{g, p}$ be a surface with $3g+p-4>0$. 
If $A$ is an abelian subgroup of $\cali(S)$, then ${\rm rank}(A)\leq 2g+p-3$ and this equality is attained for some $A$. 
The same conclusion holds for abelian subgroups of $\calk(S)$.
\end{prop}

%%%%%%%%%%%%%%%%%%%%%%%%%%%%%%%%%%%%%%%%%%%

\subsection{Characterization of twisting elements}

For a group $\Gamma$, we denote by $Z(\Gamma)$ the center of $\Gamma$.
For an element $x$ of $\Gamma$, we denote by $Z_{\Gamma}(x)$ the centralizer of $x$ in $\Gamma$.
The conclusions in Lemmas \ref{lem-cha-i}--\ref{lem-cha} for closed surfaces are announced in \cite{farb-ivanov} and are proved in \cite{v-t}.

\begin{lem}\label{lem-cha-i}
Let $S=S_{g, p}$ be a surface with $g\geq 1$ and $|\chi(S)|\geq 3$, and let $\Gamma$ be a finite index subgroup of $\cali(S)$.
Pick $x\in \Gamma$.
If $x$ is a non-zero power of either the Dehn twist about a separating curve in $S$ or the BP twist about a BP in $S$, then
\begin{enumerate}
\item[(a)] $Z(Z_{\Gamma}(x))$ is isomorphic to $\mathbb{Z}$; and
\item[(b)] $x$ is contained in a subgroup of $\Gamma$ isomorphic to $\mathbb{Z}^{2g+p-3}$.
\end{enumerate}
\end{lem}

\begin{proof}
We note that by Theorem \ref{thm-pure}, any element of $\cali(S)$ is pure in the sense of Ivanov \cite{iva-subgr}.
We put $Z(x)=Z_{\Gamma}(x)$.

We first assume that $x$ is a non-zero power of the BP twist about a BP $b$ in $S$.
The group $Z(x)$ is then equal to the stabilizer of $b$ in $\Gamma$. 
Since any element of $\cali(S)$ is pure, any element of $Z(x)$ fixes each curve in $b$ and each component of $S_b$ as a set.
We have the natural homomorphism
\[q\colon Z(x)\rightarrow \pmod(Q_1)\times \pmod(Q_2),\]
where $Q_1$ and $Q_2$ are the two components of $S_b$.
For each $j=1, 2$, let $q_j\colon Z(x)\rightarrow \pmod(Q_j)$ denote the composition of $q$ with the projection onto $\pmod(Q_j)$.
If $Q_1$ is a pair of pants, then $q_1(Z(x))$ is trivial because it fixes each curve in $b$.
Otherwise, the center of $q_1(Z(x))$ is trivial because any element in it fixes any separating curve $a$ in $Q_1$ such that the two boundary components of $Q_1$ corresponding to $b$ are contained in a component of $(Q_1)_a$.
The same property holds if $Q_1$ is replaced with $Q_2$.
It follows that the center of $q(Z(x))$ is trivial. 
The center of $Z(x)$ is thus equal to $\ker q$, which is contained in the cyclic group generated by the BP twist about $b$. 
Condition (a) is obtained.
Choose a simplex $\sigma$ of $\calt(S)$ of maximal dimension containing $b$.
The group $D_{\sigma}\cap \Gamma$ is of rank $2g+p-3$.
Condition (b) thus follows.

We next assume that $x$ is a non-zero power of the Dehn twist about a separating curve $\alpha$ in $S$.
The group $Z(x)$ is then equal to the stabilizer of $\alpha$ in $\Gamma$.
We have the natural homomorphism
\[r\colon Z(x)\rightarrow \pmod(R_1)\times \pmod(R_2),\]
where $R_1$ and $R_2$ are the two components of $S_{\alpha}$.
For each $j=1, 2$, let $r_j\colon Z(x)\rightarrow \pmod(R_j)$ denote the composition of $r$ with the projection onto $\pmod(R_j)$.
If $R_1$ is a pair of pants, then $r_1(Z(x))$ is trivial because any element of $\cali(S)$ is pure and fixes each component of $\partial S$ as a set.
If $R_1$ is a handle, then $r_1(Z(x))$ is trivial because any element of $\cali(S)$ acts trivially on the homology group $H_1(\bar{S}, \mathbb{Z})$, where $\bar{S}$ is the closed surface obtained by attaching disks to all components of $\partial S$.
Otherwise, the center of $r_1(Z(x))$ is shown to be trivial as in the previous paragraph.
The same property holds if $R_1$ is replaced with $R_2$.
Condition (a) is then obtained as in the previous paragraph.
If $\tau$ is a simplex of $\calc_s(S)$ of maximal dimension containing $\alpha$, then $D_{\tau}\cap \Gamma$ is of rank $2g+p-3$.
Condition (b) thus follows.
\end{proof}

\begin{lem}\label{lem-cha-k}
Let $S=S_{g, p}$ be a surface with $g\geq 1$ and $|\chi(S)|\geq 3$, and let $\Gamma$ be a finite index subgroup of $\calk(S)$.
Pick $x\in \Gamma$.
If $x$ is a non-zero power of the Dehn twist about a separating curve in $S$, then
\begin{enumerate}
\item[(a)] $Z(Z_{\Gamma}(x))$ is isomorphic to $\mathbb{Z}$; and
\item[(b)] $x$ is contained in a subgroup of $\Gamma$ isomorphic to $\mathbb{Z}^{2g+p-3}$.
\end{enumerate}
\end{lem}

\begin{proof}
A verbatim argument of the proof of Lemma \ref{lem-cha-i} can be applied.
\end{proof}

To prove the following lemma, let us recall reduction system graphs for simplices of $\calc(S)$, which were introduced in \cite{v}. 
Let $S=S_{g, p}$ be a surface with $3g+p-4>0$, and let $\tau$ be a simplex of $\calc(S)$. 
The {\it reduction system graph} $G(\tau)$ is then defined as follows: Vertices of $G(\tau)$ are components of $S_{\tau}$. 
Edges of $G(\tau)$ are curves in $\tau$. 
The two ends of the edge corresponding to a curve $c$ in $\tau$ are defined to be vertices corresponding to components of $S_{\tau}$ which lie in the left and right hand sides of $c$ in $S$. 
Note that $G(\tau)$ may have a loop. 
The reader should consult \cite{v} for basics of reduction system graphs.

\begin{lem}\label{lem-cha}
Let $S=S_{g, p}$ be a surface with $g\geq 1$ and $|\chi(S)|\geq 3$.
Pick $x\in \cali(S)$. 
Then $x$ is a non-zero power of either the Dehn twist about a separating curve in $S$ or the BP twist about a BP in $S$ if the following two conditions hold:
\begin{enumerate}
\item[(a)] $Z(Z_{\cali(S)}(x))$ is isomorphic to $\mathbb{Z}$; and
\item[(b)] $x$ is contained in a subgroup of $\cali(S)$ isomorphic to $\mathbb{Z}^{2g+p-3}$.
\end{enumerate}
\end{lem}

\begin{proof}
We follow argument in the proof of Theorem 3.5 in \cite{v-t}, where closed surfaces are dealt with. 
We assume that an element $x$ of $\cali(S)$ satisfies conditions (a) and (b). 
Condition (a) implies that $x$ is not neutral.
Condition (b) implies that $x$ is a reducible element of infinite order because we have $2g+p-3\geq 2$. 
Let $\sigma \in \Sigma(S)$ be the canonical reduction system for the cyclic group $\langle x\rangle$ generated by $x$. 
Since $\sigma$ is fixed by any element in the normalizer of $\langle x\rangle$ in $\mod(S)$, condition (a) implies that if $\sigma$ contains a curve of either a-type or b-type, then $x$ is a non-zero power of either the Dehn twist about a curve in $\sigma \cap V_s(S)$ or the BP twist about a BP in $S$ of two curves in $\sigma$.

We assume that $\sigma$ consists of only curves of c-type and deduce a contradiction. 
By Theorem \ref{thm-vau} (i), there exists a pseudo-Anosov component $Q$ of $S_{\sigma}$ for $x$. 
Let $\sigma_Q$ denote the set of all curves in $\sigma$ corresponding to a component of $\partial Q$. 
Let $A$ be a subgroup of $\cali(S)$ isomorphic to $\mathbb{Z}^{2g+p-3}$ and containing $x$. 
We denote by $\tau \in \Sigma(S)$ the canonical reduction system for $A$. 
Note that $\tau$ contains $\sigma$ and that $Q$ is a component of $S_{\tau}$. 
By Lemma \ref{lem-abe} (ii), (iii), each component of $S_{\tau}$ is homeomorphic to one of $S_{0, 3}$, $S_{0, 4}$, $S_{1, 1}$ and $S_{1, 2}$. 
If $Q$ were homeomorphic to $S_{1, 1}$ or $S_{1, 2}$, then at least one curve in $\sigma_Q$ would be a curve of either a-type or b-type in $\sigma$. 
This is a contradiction.
Since $Q$ is a pseudo-Anosov component for $x$, it is not homeomorphic to $S_{0, 3}$.
It follows that $Q$ is homeomorphic to $S_{0, 4}$. 
Note that there is no curve in $\sigma_Q$ whose both sides are contained in $Q$ because otherwise the other curve(s) in $\sigma_Q$ would be of a-type or b-type in $\sigma$. 
Each curve of $\sigma_Q$ is a curve of either b-type or c-type in $\tau$. 
Moreover, if two curves in $\sigma_Q$ are of b-type in $\tau$, then they are not BP-equivalent in $S$.

Let us assume that there is a curve $c\in \sigma_Q$ which is a curve of c-type in $\tau$. 
One can construct a maximal tree $T$ in the reduction system graph $G(\tau)$ for $\tau$ containing the edge corresponding to $c$ because the edge is not a loop. 
We write
\[\tau =\{ b_{11},\ldots, b_{1r_1}, b_{21},\ldots, b_{2r_2},\ldots, b_{q1},\ldots, b_{qr_q}, c_1,\ldots, c_s\}\]
so that
\begin{itemize}
\item each $b_{ij}$ is of b-type in $\tau$, and the family $\{ b_{i1},\ldots, b_{iq_i}\}$ is a BP-equivalence class in $\tau$ for each $i$; and
\item each $c_k$ is of c-type in $\tau$.
\end{itemize}
For each $i$, the tree $T$ contains at least $r_i-1$ of edges corresponding to $b_{i1},\ldots, b_{ir_i}$ because otherwise $T$ could not be connected.
Let $\nu=\nu(\tau)$ and $\Omega =\Omega(\tau)$ be the numbers defined in Lemma \ref{lem-abe}. 
Since the number of edges of $T$ is equal to $\nu-1$, we obtain the inequality
\[\nu-1\geq \sum_{i=1}^q(r_i-1)+1>\sum_{i=1}^q(r_i-1)={\rm rank}(D_{\tau}\cap \cali(S)),\]
where the last equality holds by Theorem \ref{thm-vau} (i).
We also obtain the inequality
\[{\rm rank}(A)\leq {\rm rank}(D_{\tau}\cap \cali(S))+\Omega <\nu+\Omega -1\leq |\chi(S)|-1=2g+p-3,\]
where the first inequality holds by Lemma \ref{lem-abe} (iii). 
This is a contradiction.

Finally, we suppose that $\sigma_Q$ consists of curves of b-type in $\tau$. 
We write $\tau$ as in the previous paragraph. 
For each curve $\alpha$ in $\sigma_Q$, we choose any curve $\beta_{\alpha}$ in $\tau$ BP-equivalent to $\alpha$.
If we cut $S$ along all $\beta_{\alpha}$ for $\alpha \in \sigma_Q$, then $S$ is decomposed into two connected components.
It follows that for any maximal tree $T$ in $G(\tau)$, there exists a curve $\gamma$ in $\sigma_Q$ such that $T$ contains all edges corresponding to curves in $\tau$ BP-equivalent to $\gamma$. 
We may assume $\gamma =b_{11}$. 
We then obtain the inequality
\[\nu-1\geq r_1+\sum_{i=2}^q(r_i-1)>\sum_{i=1}^q(r_i-1)={\rm rank}(D_{\tau}\cap \cali(S))\]
and deduce a contradiction as in the previous paragraph.
\end{proof}

%%%%%%%%%%%%%%%%%%%%%%%%%%%%%%%%%%%%%%%%%%%%

\subsection{Computation of commensurators}

The argument of this subsection has already appeared in many works to compute commensurators of mapping class groups and their subgroups and to describe injective homomorphisms between those groups (see e.g., \cite{be-m}, \cite{bm-ar}, \cite{bm}, \cite{bm-add}, \cite{farb-ivanov}, \cite{irmak1}, \cite{irmak2}, \cite{irmak-ns}, \cite{iva-aut}, \cite{kork-aut}, \cite{mv} and \cite{sha}). 
It is outlined as follows: To an injective homomorphism $f$ between $\cali(S)$ and $\calk(S)$, we associate a superinjective map $\phi$ between $\calt(S)$ and $\calc_s(S)$ by using the characterization of twisting elements in Lemmas \ref{lem-cha-i}--\ref{lem-cha}. 
Applying the result that $\phi$ is induced by an element $\gamma_0$ of $\mod^*(S)$, we conclude that the homomorphism $f$ is the conjugation by $\gamma_0$.

\begin{prop}\label{prop-homo-si}
Let $S=S_{g, p}$ be a surface with $g\geq 1$ and $|\chi(S)|\geq 3$.
Let $(G, X)$ be either $(\cali(S), \calt(S))$ or $(\calk(S), \calc_s(S))$. 
For any finite index subgroup $\Gamma$ of $G$ and any injective homomorphism $f\colon \Gamma \rightarrow G$, there exists a unique superinjective map $\phi \colon X\rightarrow X$ such that for each vertex $v$ of $X$, we have
\[f(T_v \cap \Gamma)<T_{\phi(v)},\]
where for each vertex $u$ of $X$, if $u\in V_s(S)$, then $T_u$ denotes the group generated by $t_u$, and otherwise $T_u$ denotes the group generated by the BP twist about the BP $u$. 
Moreover, if $f(\Gamma)$ is of finite index in $G$, then $\phi$ is an isomorphism.
\end{prop}

To prove this proposition, we use the following lemma.
For a group $H$, we denote by ${\rm rank}(H)$ the supremum of the ranks of finitely generated, free abelian subgroups in $H$.

\begin{lem}\label{lem-rank}
Let $A$ and $B$ be groups with ${\rm rank}(A)={\rm rank}(B)<\infty$, and suppose that any abelian subgroup of $B$ is finitely generated.
Let $\eta \colon A\rightarrow B$ be an injective homomorphism.
If $a$ is an element of $A$ contained in a finitely generated, free abelian subgroup of $A$ with its rank equal to ${\rm rank}(A)$, then we have the inequality
\[{\rm rank}(Z(Z_B(\eta(a))))\leq {\rm rank}(Z(Z_A(a))).\]
\end{lem}

\begin{proof}
This lemma is essentially verified in Lemma 5.2 of \cite{irmak1}.
Choose a finitely generated, free abelian subgroup $A_0$ of $A$ with $a\in A_0$ and ${\rm rank}(A_0)={\rm rank}(A)$.
We put $Z=Z(Z_B(\eta(a)))$ and $C=\eta(A_0)\cap Z$ and define $D$ as the subgroup of $B$ generated by $\eta(A_0)$ and $Z$.
Any of $Z$, $C$ and $D$ is abelian and is therefore fintiely generated by assumption.
We have the short exact sequence
\[0\rightarrow C\rightarrow \eta(A_0)\oplus Z\rightarrow D\rightarrow 0.\]
It gives the equality
\[{\rm rank}(\eta(A_0))+{\rm rank}(Z)={\rm rank}(C)+{\rm rank}(D).\]
The inequality
\[{\rm rank}(B)={\rm rank}(A)={\rm rank}(\eta(A_0))\leq {\rm rank}(D)\leq {\rm rank}(B)\]
implies the equality ${\rm rank}(Z)={\rm rank}(C)$.
On the other hand, since $C$ is contained in $Z(Z_{\eta(A)}(\eta(a)))=\eta(Z(Z_A(a)))$, we have the inequality ${\rm rank}(C)\leq {\rm rank}(Z(Z_A(a)))$. 
The lemma thus follows.
\end{proof}

\begin{proof}[Proof of Proposition \ref{prop-homo-si}]
We first assume $(G, X)=(\cali(S), \calt(S))$.
Let $x\in \Gamma$ be a non-zero power of either the Dehn twist about a separating curve in $S$ or the BP twist about a BP in $S$.
We now check conditions (a) and (b) in Lemma \ref{lem-cha} for $f(x)$.
We put $Z=Z(Z_{\cali(S)}(f(x)))$.
By Lemma \ref{lem-cha-i},
\begin{itemize}
\item $Z(Z_{\Gamma}(x))$ is isomorphic to $\mathbb{Z}$; and
\item $f(x)$ is contained in a subgroup of $\cali(S)$ isomorphic to $\mathbb{Z}^{2g+p-3}$.
\end{itemize}
By Lemma \ref{lem-rank}, the rank of $Z$ is at most one, and it is equal to one because the group generated by $f(x)$ is contained in $Z$.
The group $Z$ is isomorphic to $\mathbb{Z}$ because it is torsion-free by Theorem \ref{thm-pure} and is finitely generated by Theorem A in \cite{blm}.
We now apply Lemma \ref{lem-cha} and conclude that $f(x)$ is a non-zero power of either the Dehn twist about a separating curve in $S$ or the BP twist about a BP in $S$.
We can define a map $\phi \colon V_t(S)\rightarrow V_t(S)$ so that the inclusion $f(T_v \cap \Gamma)<T_{\phi(v)}$ holds for any $v\in V_t(S)$.

For each vertex $v$ of $\calt(S)$ and each non-zero integer $n$, we define $T_v^n$ as the subgroup of $T_v$ generated by $s^n$, where $s$ is a generator of $T_v$.
Simplicity and superinjectivity of $\phi$ follows from the fact that for any two distinct vertices $u$, $v$ of $\calt(S)$ and for any non-zero integers $m$, $n$, the group generated by $T_u^m$ and $T_v^n$ is abelian if and only if $u$ and $v$ are adjacent in $\calt(S)$.
Uniqueness of $\phi$ is obvious.
By using $f^{-1}$, we can prove that $\phi$ is an isomorphism if $f(\Gamma)$ is of finite index in $G$.

We next assume $(G, X)=(\calk(S), \calc_s(S))$.
Let $y\in \Gamma$ be a non-zero power of the Dehn twist about a separating curve in $S$.
Along the same argument as in the first paragraph of the proof, we can show that $f(y)$ is a non-zero power of either the Dehn twist about a separating curve in $S$ or the BP twist about a BP in $S$.
Proposition \ref{prop-jo-bp} implies that the latter case is impossible.
The rest of the proof is the same as that in the case $(G, X)=(\cali(S), \calt(S))$.
\end{proof}

We can also prove the following proposition along verbatim argument, which will not be used in the sequel.

\begin{prop}
Let $S=S_{g, p}$ be a surface with $g\geq 1$ and $|\chi(S)|\geq 3$.
For any finite index subgroup $\Gamma$ of $\calk(S)$ and any injective homomorphism $f\colon \Gamma \rightarrow \cali(S)$, there exists a unique superinjective map $\phi \colon \calc_s(S)\rightarrow \calt(S)$ such that for each $v\in V_s(S)$, we have
\[f(T_v \cap \Gamma)<T_{\phi(v)}.\]
\end{prop}

\begin{proof}[Proof of Theorem \ref{thm-g-1}]
Assertion (i) follows from Theorems \ref{thm-cc} and \ref{thm-g-1-si}.
We prove assertion (ii).
Let $\Gamma$ be a finite index subgroup of $\cali(S)$, and let $f\colon \Gamma \rightarrow \cali(S)$ be an injective homomorphism. 
By Proposition \ref{prop-homo-si}, there exists a superinjective map $\phi \colon \calt(S)\rightarrow \calt(S)$ such that for any vertex $v$ of $\calt(S)$, we have $f(T_v \cap \Gamma)<T_{\phi(v)}$. 
Assertion (i) shows that $\phi$ is induced by an element $\gamma_0$ of $\mod^*(S)$. 
Pick $\gamma \in \Gamma$. 
For each separating curve $a$ in $S$, we then have
\begin{align*}
f(T_{\gamma a}\cap \Gamma)&=f(\gamma T_a\gamma^{-1} \cap \Gamma)=f(\gamma)f(T_a \cap \Gamma)f(\gamma)^{-1}\\
&<f(\gamma)T_{\gamma_0a} f(\gamma)^{-1}=T_{f(\gamma)\gamma_0a},
\end{align*}
and thus $\gamma_0\gamma a=f(\gamma)\gamma_0a$.
For any non-separating curve $c$ in $S$, we can find two separating curves $a_1$, $a_2$ in $S$ such that they are disjoint from $c$ and fill $S_c$.
Since $\gamma^{-1}\gamma_0^{-1}f(\gamma)\gamma_0$ fixes $a_1$ and $a_2$, it also fixes $c$.
It follows that $\gamma^{-1}\gamma_0^{-1}f(\gamma)\gamma_0$ is the neutral element.
The equality $f(\gamma)=\gamma_0\gamma \gamma_0^{-1}$ thus holds.
\end{proof}

\begin{proof}[Proof of Theorem \ref{thm-tor-comm}]
Let
\[\pi \colon \mod^*(S)\rightarrow \aut(\calt(S))\]
be the natural homomorphism.
By Theorem \ref{thm-cc}, Theorem \ref{thm-g-1-si} and Theorem \ref{thm-g-2} (i), $\pi$ is surjective.
Following the last part in the proof of Theorem \ref{thm-g-1}, we can show that any element of $\mod^*(S)$ fixing any separating curve in $S$ is neutral.
Injectivity of $\pi$ then follows.
Assertion (i) is proved.

We prove assertion (ii).
Let
\[{\bf i}\colon \mod^*(S)\rightarrow \comm(\cali(S))\]
be the homomorphism defined by conjugation.
Surjectivity of ${\bf i}$ can be shown along argument in the proof of Theorem \ref{thm-g-1} (ii).
Let $\Gamma$ be a finite index subgroup of $\cali(S)$.
If $\gamma_0$ is an element of $\mod^*(S)$ with $\gamma_0 \gamma \gamma_0^{-1}=\gamma$ for any $\gamma \in \Gamma$, then $\gamma_0$ fixes any separating curve in $S$ because for any separating curve $a$ in $S$, there exists a non-zero integer $n$ with $t_a^n\in \Gamma$ and thus $\gamma_0t_a^n\gamma_0^{-1}=t_a^n$.
We therefore conclude that $\gamma_0$ is the neutral element and that ${\bf i}$ is injective.
\end{proof}

Theorem \ref{thm-jo-comm} can be verified in a similar manner by using Theorem \ref{thm-g-1-s} and Theorem \ref{thm-g-2} (ii).

%%%%%%%%%%%%%%%%%%%%%%%%%%%%%%%%%%%%%%%%%%%%%

\section{Commensurators of torus braid groups}\label{sec-torus}

Let $S$ be a surface and $n$ a positive integer. 
Choose $n$ base points $p_1,\ldots, p_n$ in $S$. 
We denote by $E(n, S)$ the space of embeddings of the set of $n$ points, $\{ 1,\ldots, n\}$, into $S$ with the compact-open topology. 
The {\it pure braid group of $n$-strands} on $S$, denoted by $PB_n(S)$, is defined to be the fundamental group $\pi_1(E(n, S))$. 
We refer the reader to \cite{birman-braid}, \cite{birman} and \cite{pr} for basic facts on braid groups on surfaces.

Let $T$ denote the closed torus, and let $S=S_{1, n}$ be a surface of genus one with $n$ boundary components. 
By attaching disks to all components of $\partial S$, we obtain the Birman exact sequence
\[PB_n(T)\stackrel{j}{\rightarrow}\pmod(S)\stackrel{\pi}{\rightarrow}\mod(T)\rightarrow 1.\]
A description of the homomorphism $j$ shows that $\ker \pi$ is equal to $\cali(S)$. 
On the other hand, $\ker j$ is equal to the center of $PB_n(T)$, denoted by $Z$, which is isomorphic to $\mathbb{Z}^{2}$ (see Proposition 4.2 in \cite{pr} for a precise description). 
We refer to Chapter 4 of \cite{birman} and Section 2.8 of \cite{iva-mcg} for details of the Birman exact sequence. 
Note that $PB_1(T)$ is isomorphic to $\pi_1(T)\simeq \mathbb{Z}^{2}$.

In this final section, we describe the commensurator of the braid group of $n$-strands on the torus $T$ with $n\geq 2$. 
The following computation of commensurators of central extensions is discussed in Section 3 of \cite{lm} in a general framework, and the reader should consult the reference for more details. 
We note that the automorphism groups of (pure) braid groups on $T$ are described in \cite{zhang}. 
To carry out the computation in \cite{lm}, we need the following lemma on braid groups on $T$.

\begin{lem}\label{lem-split}
For each integer $n\geq 2$, we have a homomorphism $p\colon PB_n(T)\rightarrow \mathbb{Z}^2$ such that $p$ is injective on the center $Z$ of $PB_n(T)$ and the image $p(Z)$ is the subgroup of $\mathbb{Z}^2$ generated by $(n, 0)$ and $(0, n)$.
\end{lem}

\begin{proof}
It is known that the braid group $B_n(T)$ of $n$-strands on $T$ admits the following presentation (see Theorem 1.2 in \cite{bellingeri}):
\begin{itemize}
\item Generators: $\sigma_1,\ldots, \sigma_{n-1}$, $a$, $b$.
\item Braid relations:
\begin{enumerate}
\item[(BR1)] $\sigma_i\sigma_j=\sigma_j\sigma_i$ for each $i$, $j$ with $|i-j|\geq 2$;
\item[(BR2)] $\sigma_i\sigma_{i+1}\sigma_i=\sigma_{i+1}\sigma_i\sigma_{i+1}$ for each $i$ with $1\leq i\leq n-2$; 
\end{enumerate}
\item Mixed relations:
\begin{enumerate}
\item[(R1)] $a\sigma_i=\sigma_ia$ and $b\sigma_i=\sigma_ib$ for each $i$ with $1<i\leq n-1$;
\item[(R2)] $\sigma_1^{-1}a\sigma_1^{-1}a=a\sigma_1^{-1}a\sigma_1^{-1}$ and $\sigma_1^{-1}b\sigma_1^{-1}b=b\sigma_1^{-1}b\sigma_1^{-1}$;
\item[(R3)] $\sigma_1^{-1}a\sigma_1^{-1}b=b\sigma_1^{-1}a\sigma_1$
\item[(TR)] $[a, b^{-1}]=\sigma_1\cdots \sigma_{n-2}\sigma_{n-1}^2\sigma_{n-2}\cdots \sigma_1$.
\end{enumerate}
\end{itemize}
We can define a homomorphism $p\colon B_n(T)\rightarrow \mathbb{Z}^{2}$ by putting $p(\sigma_i)=(0, 0)$ for each $i=1,\ldots, n-1$ and putting $p(a)=(1, 0)$ and $p(b)=(0, 1)$. 
Note that $PB_n(T)$ is a normal subgroup of $B_n(T)$ and that the quotient group is isomorphic to the symmetric group of $n$ letters. 
It follows from geometric description of $\sigma$, $a$ and $b$ in Section 2.2 of \cite{bellingeri} and of the center of $B_n(T)$ in Proposition 4.2 of \cite{pr} that the two elements
\[a(\sigma_1^{-1}a\sigma_1^{-1})(\sigma_2^{-1}\sigma_1^{-1}a\sigma_1^{-1}\sigma_2^{-1})\cdots (\sigma_{n-1}^{-1}\sigma_{n-2}^{-1}\cdots \sigma_1^{-1}a\sigma_1^{-1}\cdots \sigma_{n-2}^{-1}\sigma_{n-1}^{-1})\]
and
\[b(\sigma_1^{-1}b\sigma_1^{-1})(\sigma_2^{-1}\sigma_1^{-1}b\sigma_1^{-1}\sigma_2^{-1})\cdots (\sigma_{n-1}^{-1}\sigma_{n-2}^{-1}\cdots \sigma_1^{-1}b\sigma_1^{-1}\cdots \sigma_{n-2}^{-1}\sigma_{n-1}^{-1})\]
generate the center of $B_n(T)$, which is also equal to the center $Z$ of $PB_n(T)$. 
The image $p(Z)$ is therefore generated by $(n, 0)$ and $(0, n)$.
\end{proof}

Let $S=S_{1, n}$ be a surface with $n\geq 2$. 
Lemma \ref{lem-split} implies that the exact sequence
\[1\rightarrow Z\rightarrow PB_n(T)\rightarrow \cali(S)\rightarrow 1\]
virtually split. 
More precisely, the exact sequence
\[1\rightarrow Z\rightarrow p^{-1}(p(Z))\rightarrow j(p^{-1}(p(Z)))\rightarrow 1\]
splits. 
We can thus find a finite index subgroup of $PB_n(T)$ isomorphic to $\mathbb{Z}^2\times \Gamma$, where $\Gamma$ is a finite index subgroup of $\cali(S)$. 
Using the fact that the center of any finite index subgroup of $\cali(S)$ is trivial, we obtain the split exact sequence
\[1\rightarrow {\rm Tv}\rightarrow \comm(PB_n(T))\rightarrow \comm(\Gamma)\rightarrow 1.\]
The group ${\rm Tv}$, called the transvection subgroup, fits into the split exact sequence
\[1\rightarrow \lim_iH^1(\Gamma_i, \mathbb{Z}^2)\rightarrow {\rm Tv}\rightarrow \comm(\mathbb{Z}^2)\rightarrow 1,\]
where $\lim_i$ is the direct limit taken over all finite index subgroups $\Gamma_i$ of $\Gamma$. 
As a conclusion, we have
\[\comm(PB_n(T))\simeq \comm(\cali(S_{1, n}))\ltimes (GL(2, \mathbb{Q})\ltimes H),\]
where we put $H=\lim_iH^1(\Gamma_i, \mathbb{Z}^2)$. 
The group $H$ is abelian and torsion-free because so is $H^1(\Gamma_i, \mathbb{Z}^2)$ for any finite index subgroup $\Gamma_i$ of $\Gamma$.
Along the proof of Proposition 4 in \cite{lm}, we can show that $H$ is divisible, that is, for any $x\in H$ and any positive integer $m$, there exists $y\in H$ with $x=y^m$.
It follows that $H$ is a vector space over $\mathbb{Q}$.
Since $\cali(S_{1, 2})$ is isomorphic to $\pi_1(S_{1, 1})$ and to the free group of rank two, there exists a surjective homomorphism from $\cali(S)$ onto the free group of rank two.
Using this fact, we can show that $H$ is isomorphic to the countably infinite dimensional vector space $\mathbb{Q}^{\infty}$ over $\mathbb{Q}$ (see the aforementioned reference). 
Theorem \ref{thm-tor-comm} (ii) shows that $\comm(\cali(S_{1, n}))$ is naturally isomorphic to $\mod^*(S_{1, n})$ when $n\geq 3$.

%%%%%%%%%%%%%%%%%%%%%%%%%%%%%%%%%%%%%%%%%%%

\appendix

\section{Pureness of elements of the Torelli group}\label{sec-app}

The following theorem is stated in Theorem 3.1 of \cite{iim} for the pure braid group on a surface $S$, which is contained in the Torelli group $\cali(S)$.

\begin{thm}\label{thm-pure}
Let $S=S_{g, p}$ be a surface with $g\geq 1$. 
Then each element of $\cali(S)$ is pure. 
Namely, for each element $f\in \cali(S)$, there exist $\sigma \in \Sigma(S)\cup \{ \emptyset \}$ and representatives $F$ of $f$ and $C$ of $\sigma$ such that
\begin{itemize}
\item $F(C)=C$;
\item $F$ does not exchange components of $C$ and components of the surface $S_C$ obtained by cutting $S$ along $C$; and
\item $F$ induces either a homeomorphism isotopic to the identity or a pseudo-Anosov homeomorphism on each component of $S_C$.
\end{itemize}
Moreover, if $\tau \in \Sigma(S)$ is fixed by $f$, then any curve in $\tau$ and any component of $S_{\tau}$ are fixed by $f$. 
\end{thm}

This theorem implies that $\cali(S)$ is torsion-free when the genus of $S$ is positive.
The proof of the theorem will be given after the following three lemmas.

\begin{lem}\label{lem-pre}
Let $S=S_{g, p}$ be a surface with $g\geq 1$ and pick $f\in \cali(S)$. 
Suppose that $\sigma \in \Sigma(S)$ is fixed by $f$. 
Choose representatives $F$ of $f$ and $C$ of $\sigma$ such that $F(C)=C$ and $F$ is the identity on $\partial S$. 
Then $F$ preserves each component of $C$.
\end{lem}

\begin{proof}
Let $\bar{S}$ denote the surface obtained by attaching disks to all components of $\partial S$. 
We define the homeomorphism $\bar{F}$ of $\bar{S}$ by extending $F$ so that $\bar{F}$ is the identity on all attached disks. 
We denote by $\calc^*(\bar{S})$ the simplicial cone over $\calc(\bar{S})$ with the cone point $\ast$. 
Let $\pi \colon \calc(S)\rightarrow \calc^*(\bar{S})$ be the simplicial map associated with the inclusion of $S$ into $\bar{S}$, where $\pi^{-1}(\{ \ast \})$ consists of all separating curves in $S$ cutting off a surface of genus zero.

Assume that there are components $c_1$, $c_2$ of $C$ with $F(c_1)=c_2$ and $c_1\neq c_2$. 
We first claim that the equality $\pi([c_1])=\pi([c_2])$ holds, where $[c]$ denotes the isotopy class of a curve $c$ in $S$. 
When $g=1$, the claim follows because any two disjoint curves in the torus are isotopic. 
Assume $g\geq 2$. 
Since $\bar{F}$ acts on $H_1(\bar{S}, \mathbb{Z})$ trivially, Theorem 1.2 of \cite{iva-subgr} implies that $\bar{F}$ fixes each element of $\pi(\sigma)$. 
We thus have $\pi([c_1])=\pi([c_2])$.

If $\pi([c_1])=\pi([c_2])=\ast$, then the set of components of $\partial S$ contained in the surface of genus zero cut off by $c_1$ and that by $c_2$ are distinct. 
Since $f([c_1])=[c_2]$, this contradicts the fact that $f$ does not exchange components of $\partial S$. 
Let us assume $\pi([c_1])=\pi([c_2])\in V(\bar{S})$. 
It then follows that $c_1$ and $c_2$ cut off a holed annulus $A$ from $S$. 
Since $f$ does not exchange components of $\partial S$, we have $F(A)=A$. 
Orient $c_1$ and $c_2$ so that they are parallel when disks are attached to all components of $A\cap \partial S$. 
It follows from $F(c_1)=c_2$ and $F(A)=A$ that the orientations of $F(c_1)$ and $c_2$ are distinct. 
If $c_1$ is non-separating in $S$, then so is $c_2$, and both $c_1$ and $c_2$ are non-zero as an element of $H_1(\bar{S}, \mathbb{Z})$. 
This contradicts the fact that $\bar{F}$ acts on $H_1(\bar{S}, \mathbb{Z})$ trivially. 
If $c_1$ is separating in $S$, then so is $c_2$. 
For each $i=1, 2$, let $R_i$ be the component of $S_{c_i}$ that does not contain $A$. 
We then have $F(R_1)=R_2$. 
In particular, $R_1$ and $R_2$ are homeomorphic and are of positive genus. 
This also contradicts the fact that $\bar{F}$ acts on $H_1(\bar{S}, \mathbb{Z})$ trivially.
\end{proof}

\begin{lem}\label{lem-ori-pre}
In the notation of Lemma \ref{lem-pre}, $F$ preserves an orientation of each component of $C$. 
Moreover, $F$ preserves each component of $S_C$.
\end{lem}

\begin{proof}
Let $c$ be a component of $C$. 
If $c$ is non-separating in $S$, then $F$ preserves an orientation of $c$ because $\bar{F}$ acts on $H_1(\bar{S}, \mathbb{Z})$ trivially, where $\bar{F}$ and $\bar{S}$ are the symbols in the proof of Lemma \ref{lem-pre}. 
Suppose that $c$ is separating in $S$. 
If $F$ reversed an orientation of $c$, then $F$ would exchange the two components of $S_c$. 
We can then deduce a contradiction as in the proof of Lemma \ref{lem-pre}. 
The latter assertion of the lemma follows from the former assertion.
\end{proof}

\begin{lem}\label{lem-id}
Let $S=S_{g, p}$ be a surface with $g\geq 1$, and pick $f\in \cali(S)$ and $\sigma \in \Sigma(S)\cup \{ \emptyset \}$ with $f\sigma =\sigma$. 
Choose representatives $F$ of $f$ and $C$ of $\sigma$ such that $F(C)=C$ and $F$ is the identity on $\partial S$. 
Let $Q$ be a component of $S_C$, and suppose that the mapping class of the homeomorphism $F_Q$ on $Q$ induced by $F$ is of finite order as an element of $\pmod(Q)$. 
Then $F_Q$ is isotopic to the identity.
\end{lem}

\begin{proof}
We prove this lemma by induction on $p$, the number of components of $\partial S$. 
When $p=0, 1$, the lemma follows from Lemma 1.6 of \cite{iva-subgr} because $f$ acts on $H_1(S, \mathbb{Z})$ trivially. 
Assume $p\geq 2$, and let $Q$ be a component of $S_C$. 
If $Q$ is a pair of pants, then $F_Q$ is isotopic to the identity because $\pmod(Q)$ is trivial. 
We thus assume that $Q$ is not a pair of pants.

We first assume that $Q$ contains a component of $\partial S$. 
By attaching a disk $D_1$ to that component of $\partial S$, one obtains the surfaces $Q_1=Q\cup D_1$ and $S_1=S\cup D_1$ with $\chi(Q_1)<0$ since $Q$ is not a pair of pants. 
Let $F_1$ be the homeomorphism of $S_1$ defined by the extention of $F$ that is the identity on $D_1$. 
Note that $C$ either determines a simplex of $\calc(S_1)$ or is empty and that $p_1(f)$ belongs to $\cali(S_1)$, where $p_1\colon \pmod(S)\rightarrow \pmod(S_1)$ is the natural homomorphism. 
Since the mapping class, denoted by $f_Q$, of $F_Q$ is of finite order, so is the mapping class, denoted by $f_{Q_1}$, of the restriction of $F_1$ to $Q_1$. 
The hypothesis of the induction implies that $f_{Q_1}$ is the identity. 
It then follows that $f_Q$ lies in the kernel of the natural homomorphism from $\pmod(Q)$ into $\pmod(Q_1)$, which is isomorphic to $\pi_1(Q_1)$ and is torsion-free. 
Therefore, $f_Q$ is the identity.

We next assume that $Q$ contains no component of $\partial S$. 
By attaching a disk $D_2$ to a component $\partial$ of $\partial S$, we obtain the surface $S_2=S\cup D_2$. 
It is then possible either that there are two components of $C$ which are isotopic in $S_2$ to each other or that there is a component of $C$ which is isotopic in $S_2$ to a component of $\partial S_2$. 
If the former is the case, then delete one of those two components of $C$. 
If the latter is the case, then delete that component of $C$. 
Otherwise, we do nothing. 
We then obtain the family, denoted by $C_0$, of essential simple closed curves in $S_2$ which are pairwise non-isotopic in $S_2$. 
We denote by $Q_0$ the component of $(S_2)_{C_0}$ containing $Q$. 
Note that the complement of the interior of $Q$ in $Q_0$ is an annulus containing $D_2$ if it is non-empty. 
Let $F_2$ denote the homemomorphism of $S_2$ defined by the extension of $F$ which is the identity on $D_2$. 
Since the mapping class of $F_Q$ is of finite order, so is the mapping class of the restriction of $F_2$ to $Q_0$, denoted by $F_{Q_0}$. 
By the hypothesis of the induction, $F_{Q_0}$ is isotopic to the identity, and thus $F_Q$ is isotopic to the identity.
\end{proof}

\begin{proof}[Proof of Theorem \ref{thm-pure}]
The latter assertion of the theorem follows from Lemmas \ref{lem-pre} and \ref{lem-ori-pre}. 
Let $C$ be a representative of the canonical reduction system $\sigma \in \Sigma(S)\cup \{ \emptyset \}$ for the cyclic group generated by $f$ with $F(C)=C$. 
We may assume that $F$ is the identity on $\partial S$. 
Lemmas \ref{lem-pre} and \ref{lem-ori-pre} imply that $F$ preserves each component of $C$ and each component of $S_C$. 
It follows from Theorem 7.16 in \cite{iva-subgr} that the mapping class of the restriction of $F$ to each component of $S_C$ is either of finite order or pseudo-Anosov. 
By Lemma \ref{lem-id}, it is either trivial or pseudo-Anosov.
\end{proof}

%%%%%%%%%%%%%%%%%%%%%%%%%%%%%%%%%%%%%%%%%%%%%

\end{document}